\documentclass [twoside,reqno, 12pt]{amsart}

\usepackage{hyperref}
\usepackage{cancel}
\usepackage{amsfonts}
\usepackage{amssymb}
\usepackage{a4}
\usepackage{color}
\usepackage{amscd,amssymb,amsfonts}

\usepackage{picture,xcolor}

\usepackage{etoolbox}

\usepackage{tikz}

\newtheorem{thm}{Theorem}[section]
\newtheorem{cor}[thm]{Corollary}
\newtheorem{lem}[thm]{Lemma}
\newtheorem{prop}[thm]{Proposition}
\newtheorem{exmp}[thm]{Example}

\newtheorem{rem}[thm]{Remark}

\theoremstyle{definition}

\numberwithin{equation}{section}

\newcommand{\C}{\mathbb{C}}

\newcommand{\N}{\mathbb{N}}

\newcommand{\R}{\mathbb{R}}

\def\tilde{\widetilde}
\def \bfo {\begin {eqnarray*} }
\def \efo {\end {eqnarray*} }
\def \ba {\begin {eqnarray*} }
\def \ea {\end {eqnarray*} }
\def \beq {\begin {eqnarray}}
\def \eeq {\end {eqnarray}}

\def \dist {\hbox{dist}}

\def \det {\hbox{det}}

\def \p {\partial}

\def\tilde{\widetilde}
\def \bfo {\begin {eqnarray*} }
\def \efo {\end {eqnarray*} }
\def \ba {\begin {eqnarray*} }
\def \ea {\end {eqnarray*} }
\def \beq {\begin {eqnarray}}
\def \eeq {\end {eqnarray}}

\def \dist {\hbox{dist}}

\def \det {\hbox{det}}

\def \p {\partial}

\begin{document}

\title[Inverse problems for semilinear Schr\"odinger equations]{Inverse problems for semilinear Schr\"odinger equations at large frequency via polynomial resolvent estimates on manifolds}

\author[Krupchyk]{Katya Krupchyk}
\address
        {K. Krupchyk, Department of Mathematics\\
University of California, Irvine\\
CA 92697-3875, USA }

\email{katya.krupchyk@uci.edu}

\author[Ma]{Shiqi Ma}

\address
       {S. Ma, School of Mathematics\\
Jilin University \\
Changchun, China}
\email{mashiqi@jlu.edu.cn}

\author[Kumar Sahoo]{Suman Kumar Sahoo}

\address
       {S. K. Sahoo, Seminar for Applied Mathematics\\
        Department of Mathematics\\ 
        ETH Z\"urich, Switzerland}
\email{susahoo@ethz.ch}

\author[Salo]{Mikko Salo}

\address
       {M. Salo, Department of Mathematics and Statistics\\
       University of Jyv\"askyl\"a, Jyv\"askyl\"a\\
       FI-40014, Finland}
\email{mikko.j.salo@jyu.fi}

\author[St-Amant]{Simon St-Amant}

\address
       {S. St-Amant, Department of Pure Mathematics and Mathematical Statistics\\ 
       University of Cambridge, Cambridge CB3 0WB, UK}
\email{sas242@cam.ac.uk}

\maketitle

\begin{abstract}
We study inverse boundary problems for semilinear Schr\"odinger equations on smooth compact Riemannian manifolds of dimensions $\ge 2$ with smooth boundary,  at a large fixed frequency. We show that certain classes of cubic nonlinearities are determined uniquely from the knowledge of the nonlinear Dirichlet--to--Neumann map at a large fixed frequency on quite general Riemannian manifolds. In particular, in contrast to the previous results available, here the manifolds need not satisfy any product structure, may have trapped geodesics, and the geodesic ray transform need not be injective. Only a mild assumption about the geometry of intersecting geodesics is required. We also establish a polynomial resolvent estimate for the Laplacian on an arbitrary smooth compact Riemannian manifold without boundary, valid for most frequencies. This estimate, along with the invariant construction of Gaussian beam quasimodes with uniform bounds for underlying constants and a stationary phase lemma with explicit control over all involved constants, constitutes the key elements in proving the uniqueness results for the considered inverse problems.

\end{abstract}

\section{Introduction and statement of results}

The anisotropic Calder\'on problem seeks to determine the electrical conductivity matrix of a medium by performing electrical measurements along its boundary, see \cite{Calderon_1980, Uhlmann_2014}.  To state the geometric version of this problem,  let $(M,g)$ be a smooth oriented compact Riemannian manifold of dimension $n\ge 2$ with smooth boundary. Associated to the metric $g$, we have the Laplace operator $-\Delta_g$ given in local coordinates by 
\[
-\Delta_g=-\sum_{j,k=1}^n |g|^{-1/2}\p_{x_j}(|g|^{1/2}g^{jk}\p_{x_k}),
\]
where $(g^{jk})=(g_{jk})^{-1}$ and $|g|=\det(g_{jk})$. Let us introduce the Cauchy data set on  $\p M$ of harmonic functions on $M$,
\[
C_g=\{(u|_{\p M}, \p_\nu u|_{\p M}): u\in C^\infty(M) \text{ such that } -\Delta_g u=0 \text{ in }M^{\text{int}}\},
\]
where $\nu$ is the unit outer normal to the boundary of $M$ and $M^{\text{int}}=M\setminus\p M$ stands for the interior of $M$. 
In the geometric version of the anisotropic Calder\'on problem, one wishes to recover the Riemannian manifold $(M,g)$ from the knowledge of the Cauchy data set $C_g$ on $\p M$. Because $C_{\psi^*g}=C_g$ for any smooth diffeomorphism $\psi:M\to M$ such that $\psi|_{\p M}=\text{Id}$, see  \cite{Lee_Uhlmann_1989}, one can only hope to recover a Riemannian manifold from $C_g$ up to this isometry. In dimension $n=2$,  this problem, with an additional obstruction arising from the conformal invariance of the Laplacian, is solved in \cite{Nachman_1996} for the case of Riemannian metrics on bounded domains in the plane, and in 
\cite{Lassas_Uhlmann_2001} for the case of Riemannian surfaces. In dimensions $n\ge 3$, the anisotropic Calder\'on problem is widely open. It is solved for real analytic manifolds in \cite{Lee_Uhlmann_1989, Lassas_Uhlmann_2001, Lassas_Taylor_Uhlmann_2003}, see also \cite{Guillarmou_Sa_Barreto_2009},  and in the $C^\infty$ case in \cite{DKSaloU_2009, DKurylevLS_2016}, for metrics in a fixed conformal class on a conformally transversally anisotropic (CTA) manifold, assuming that the geodesic X-ray transform on the transversal manifold is injective. In the latter case, the anisotropic Calder\'on inverse problem can be reduced to an inverse boundary problem of recovery of the potential in the Schr\"odinger equation from the knowledge of the Cauchy data of its solutions on the boundary. The approach of the works \cite{DKSaloU_2009, DKurylevLS_2016} relies on a construction of special solutions to the Schr\"odinger equation, called complex geometric optics solutions. The requirement that the manifold should be CTA guarantees the existence of so-called limiting Carleman weights on such manifolds, which are important to construct such solutions using the technique of Carleman estimates. However, a generic manifold of dimension $n\ge 3$ does not admit limiting Carleman weights, see \cite{Liimatainen_Salo_2012, Angulo-Ardoy_2017}.  Also, there are examples of manifolds for which the geodesic X-ray transform is non-injective, see \cite[Example 1.7]{KrLiimSalo_2022}. Thus, it is of great significance to try to remove the assumption of the injectivity of the geodesic X-ray transform on the transversal manifold in the case of CTA manifolds, or even to try to remove the assumption that the manifold is CTA altogether when solving the anisotropic Calder\'on problem. 

In this direction, the works \cite{LLLS_2021,Feizmohammadi_Oksanen_2020, FLL_2021, Krup_Uhlmann_nonlinear_mag} demonstrated that when solving inverse boundary problems for semilinear Schr\"odinger equations, no assumption on the transversal manifold is required, while still keeping the assumption that the manifold should enjoy a CTA structure, see also \cite{Krup_Uhlmann_Yan_nonlinear_mag_2021, Ma_Tzou_2021,  CFO_2022, FKianOksanen_2023}.  

The recent work \cite{Uhlmann_Wang_2021} demonstrated that the recovery of certain classes of potentials in the Schr\"odinger equation at a large fixed frequency did not require the manifold to be CTA. The work \cite{Ma_Sahoo_Salo_2022} extended the result of \cite{Uhlmann_Wang_2021} and showed that the manifold only needed to be non-trapping, and the geodesic X-ray transform needed to be stably invertible and continuous. We refer to the recent work \cite{St-Amant_2024}, where inverse problems of the recovery of connections from the Dirichlet--to--Neumann maps at a large fixed frequency were studied.

Motivated by \cite{Uhlmann_Wang_2021, Ma_Sahoo_Salo_2022}, the purpose of this paper is to show that an inverse boundary problem for a semilinear Schr\"odinger equation for a certain class of cubic nonlinearities, at a large fixed frequency, can be solved in quite general geometries. Specifically, we do not require our manifold to be non-trapping and we do not need any assumption related to the geodesic X-ray transform on it. Only a mild assumption about the geometry of geodesics is needed.  

To state the inverse problem, let $\lambda\ge 0$ be a frequency and let us consider the Dirichlet problem for the semilinear Schr\"odinger operator, 
\begin{equation}
\label{eq_1_1}
\begin{cases}
-\Delta_g u -\lambda^2 u+qu^3=0 &\text{in}\quad M^{\text{int}},\\
u=f& \text{on}\quad \p M,
\end{cases}
\end{equation} 
where $q\in C^{0,\alpha}(M)$ with some $0<\alpha<1$.  Here and in what follows $C^{k,\alpha}(M)$,  $k\in \N$,  stands for the H\"older space on $M$.
Assume that 
\begin{itemize}
\item[(A)] $\lambda^2$ is not a Dirichlet eigenvalue of $-\Delta_g$ on $M$. 
\end{itemize}

It follows from \cite[Theorem B.1]{Krup_Uhlmann_nonlinear_mag},  see also \cite[Proposition 1]{Feizmohammadi_Oksanen_2020} and \cite[Proposition 2.1]{LLLS_2021},
 that under the assumption (A), there is $\delta>0$ and $C>0$ such that when $f\in B_\delta (\p M):=\{f\in C^{2,\alpha}(\p M): \|f\|_{C^{2,\alpha}(\p M)}<\delta\}$, the problem \eqref{eq_1_1} has a unique solution $u=u_f\in C^{2,\alpha}(M)$ satisfying $\|u\|_{C^{2,\alpha}(M)}<C\delta$. Associated to the problem \eqref{eq_1_1}, we define the Dirichlet--to--Neumann map at the frequency $\lambda$ as follows
\[
\Lambda_q^\lambda f=\p_\nu u_f|_{\p M}, 
\]
where $f\in B_\delta (\p M)$.  The inverse problem that we are interested in is to determine $q$ from the knowledge of the Dirichlet--to--Neumann map $\Lambda_q^\lambda$ for a large but fixed frequency $\lambda$. 

First, motivated by \cite[Definition 1.2]{DKuLLS_2020} and \cite[Definition 2.1]{KrLiimSalo_2022}, we impose the following geometric condition on the manifold $(M,g)$:
\begin{itemize}
\item[(H)] for every point $x_0\in M^{\text{int}}$, there are two non-tangential unit speed geode\-sics, passing through $x_0$,  such that they do not have self-intersections at  $x_0$ and $x_0$ is  their only point of intersection. 
\end{itemize}

Our first main result is as follows, where $|\cdot|$ denotes the Lebesgue measure on $\mathbb{R}$.
\begin{thm}
\label{thm_main}
Let $(M,g)$ be a smooth compact oriented Riemannian manifold of dimension $n\ge 2$ with smooth boundary, satisfying the condition (H). Let $0<\alpha<1$ and let $q_1,q_2\in C^{0,\alpha}(M)$.
Then for any $\delta>0$, there exists a set $J\subset [1,\infty)$, $J=J(M,g,\delta)$, satisfying $|J|\le \delta$, and $\lambda_0=\lambda_0(M,g, \delta,  q_1-q_2)>0$ such that if for some $\lambda\ge \lambda_0$, $\lambda\notin J$,  and $\lambda$ satisfying the assumption (A), we have 
$\Lambda_{q_1}^\lambda=\Lambda_{q_2}^\lambda$ then $q_1=q_2$ in $M$. 
\end{thm}

\begin{rem}
In contrast to most prior results concerning elliptic partial differential equations in dimensions $n \geq 3$, the manifolds considered in Theorem \ref{thm_main} are not limited to having any specific product structure. They may contain trapped geodesics, and the geodesic X-ray transform is not required to be injective.
\end{rem}

Let us proceed to make some comments regarding the geometric assumption (H). 
\begin{exmp}
If the manifold $(M,g)$ is simple, i.e.\ a compact simply connected Riemannian manifold with strictly convex boundary so that no geodesic has conjugate points, then the geometric assumption (H) holds. This follows directly from the fact that  the exponential map $\exp_x$ is a diffeomorphism onto $M$ for any $x \in M$ \cite{PSU_book}, and hence any two geodesics $\exp_x(tv)$ and $\exp_x(tw)$ where $v \neq w$, $|v| = |w| = 1$, will satisfy (H).
\end{exmp}

\begin{exmp}
If the manifold $(M,g)$ satisfies the strict Stefanov--Uhlmann regularity condition at every point $(x_0,\xi_0)\in S^*M^{\mathrm{int}}$, see \cite[Definition 1.8]{KrLiimSalo_2022}, then the assumption (H) holds, see  \cite[Lemma 3.1]{DKuLLS_2020}. As an example of such manifolds, consider $M$ to be the closure of a neighborhood of a geodesic arc from the north pole to the south pole of the unit sphere $\mathbb{S}^3\subset \R^4$, see \cite{DKuLLS_2020} and \cite[Example 1.10]{KrLiimSalo_2022}. Note that the manifold $M$ contains conjugate points so it is not simple.
\end{exmp}

\begin{exmp}
Let $M = \mathbb{S}^1 \times [0,a]$, $a > 0$, be a cylinder with its usual flat metric $g$.  The geodesics on $M$ are straight lines, circular cross sections, and helices that wind around the cylinder. It is shown in \cite[Appendix A]{KrLiimSalo_2022} that $(M,g)$ satisfies the assumption (H). Furthermore, the geodesic X-ray transform is not invertible on $M$, since the kernel contains functions of the form $f(e^{it}, s) = h(s)$ where $h \in C^{\infty}_0((0,a))$ integrates to zero over $[0,a]$, and $(M,g)$ has trapping. 
\end{exmp}

We shall next propose a geometric condition on the manifold $(M,g)$ which guarantees that the assumption (H) holds for almost every point in $M^{\text{int}}$, and which might be of independent interest. To motivate the proposed geometric condition, we first note that the fact that all the geodesics on spheres intersect twice can be attributed to the high order of conjugacy of the sphere $\mathbb{S}^n$. To state our result, we shall introduce some definitions, following \cite[Section 10]{Lee_Riem_mfld_2018}, \cite[Chapter 5, Section 3]{do_Carmo_1992}. Given a geodesic $\gamma: I\to M$, and two points  $p=\gamma(a)$,  $q=\gamma(b)$, $a,b\in I$, we say that $p$ and $q$ are conjugate along $\gamma$ if there is a Jacobi field along $\gamma$ vanishing at $t=a$ and $t=b$ but not identically zero. The order of conjugacy of $p$ relative to $q$ along $\gamma$ denoted $\mathrm{conj}_{\gamma}(p, q)$, is the dimension of the space of Jacobi fields along $\gamma$ that vanish at $a$ and $b$. Note that $\mathrm{conj}_\gamma(p, q) \leq n-1$,  and that equality is achieved on the sphere $\mathbb{S}^n$ by taking $p$ and $q$ to be antipodal points, and $\gamma$ any great circle connecting them, see  \cite[Section 10, page 299]{Lee_Riem_mfld_2018}. We define the order of conjugacy of $p \in M$ as the maximal value of $\mathrm{conj}_\gamma(p, q)$ over all other points $q \in M$ and geodesics $\gamma$ connecting $p$ to $q$. Similarly, the order of conjugacy of $M$ is the maximal order of conjugacy of all points in $M$. We have the following result.

\begin{thm}\label{thm:intersection}
Let $(M,g)$ be a smooth compact Riemannian manifold of dimension $n\ge 3$ with smooth boundary.  Suppose that the order of conjugacy of $M$ is at most $n-3$. Then the assumption (H) holds for almost all $x \in M^{\mathrm{int}}$. 
\end{thm}

\begin{exmp}
Let $(M_j,g_j)$ be a smooth compact Riemannian manifold of dimension $n_j\ge 1$ without boundary, $j=1,2,3$.   If $M_1$ has order of conjugacy $k_1$ and $M_2$ has order of conjugacy $k_2$, then $M_1 \times M_2$ has order of conjugacy $k_1 + k_2$. Indeed, if $J$ is a Jacobi field vanishing at $p = (p_1, p_2) \in M_1 \times M_2$, then
\[
J(t) = d (\exp_{(p_1, p_2)})_{(tv_1, tv_2)}(tw_1, tw_2) = d(\exp_{p_1})_{tv_1}(tw_1) \oplus d(\exp_{p_2})_{tv_2}(tw_2)
\]
decomposes as the direct sum of Jacobi fields $J_1$ and $J_2$ on $M_1$ and $M_2$ vanishing at $p_1$ and $p_2$ respectively, see \cite[Proposition 10.10]{Lee_Riem_mfld_2018}. Hence, $J$ vanishes at $q = (q_1, q_2)$ if and only if $J_1$ and $J_2$ vanish at $q_1$ and $q_2$ respectively. Since the manifold $M_j$ has order of conjugacy at most $n_j-1$, the product of three manifolds $M_1\times M_2\times M_3$ has order of conjugacy at most $n_1+n_2+n_3-3$, where $n_1+n_2+n_3$ is the dimension of $M_1\times M_2\times M_3$. Hence, any smooth compact Riemannian manifold with smooth boundary that can be isometrically embedded as a codimension 0 submanifold (also known as a regular domain) of a product of three manifolds $M_1\times M_2\times M_3$ satisfies the assumption of Theorem \ref{thm:intersection}. 
\end{exmp}

Next, we let $0 < \alpha < 1$ be fixed and following \cite{Ma_Sahoo_Salo_2022}, for any nonzero $p\in C^{0,\alpha}(M)$, we introduce the frequency function $N(p)$ by 
\[
N(p)=\frac{\|p\|_{C^{0,\alpha}(M)}}{\|p\|_{L^\infty(M)}},
\]
and we set $N(p)=0$ if $p=0$. When $B>0$, we define the set $\mathcal{A}(B)$ of admissible perturbations by 
\[
\mathcal{A}(B) = \{p\in  C^{0,\alpha}(M): N(p)\le B \}.
\]

Note that for any $q_1,q_2\in C^{0,\alpha}(M)$, there exists a constant $B>0$ such that $q_1-q_2\in \mathcal{A}(B)$. Recall that in Theorem \ref{thm_main}, the frequency $\lambda_0$ depends on the difference  $q_1-q_2$.  Specifically, it follows from the proof of Theorem \ref{thm_main} that the value of $\lambda_0$ is contingent on the following factors:
\begin{itemize}
\item the constant $B>0$ such that $q_1-q_2\in \mathcal{A}(B)$, 
\item the specific point $x_0\in M$ where $\sup_{x\in M}|q_1(x)-q_2(x)|$ is reached,
\item the choice of two non-tangential geodesics passing through $x_0$ and fulfilling the assumption (H), including their lengths and the angle between them. 
\end{itemize}

Our next result aims to improve Theorem \ref{thm_main} by eliminating the dependency of the frequency $\lambda_0$ on a specific point $x_0$ and the particular choice of two non-tangential geodesics passing through $x_0$ and satisfying assumption (H). To achieve this, we will impose the following geometric condition on the manifold $(M, g)$, which will replace the condition (H). To state this condition, we let $\p_+SM$ be a part of the boundary of the unit sphere bundle $SM$, consisting of inward vectors, see Section \ref{sec_solutions_Helmholtz_uniform_constants} below. For  $(x, w)\in \p_+SM$,   we let  $\gamma=\gamma_{x, w}:[0,\tau(x,w)]\to M$ be the unit speed geodesic such that $\gamma(0)=x$, and $\dot{\gamma}(0)=w$.

\begin{itemize}
\item[(H1)] There are constants $T > 0$,  $0<\theta_0 <\pi/2$, $0 < r < \mathrm{Inj}(M)/2$, and $c_0 > 0$, such that for almost every point $x_0\in M^{\text{int}}$, there exist two non-tangential unit-speed geodesics $\gamma=\gamma_{x_1,w_1}:[0,\tau(x_1,w_1)]\to M$ and $\eta=\eta_{x_2,w_2}:[0,\tau(x_2,w_2)]\to M$, $(x_1,w_1), (x_2,w_2)\in \p_+SM$ such that 
\begin{itemize}
\item[(i)]
$\gamma(t_0)=\eta(\tau_0)=x_0$, $t_0\in (0,\tau(x_1,w_1))$, $\tau_0\in(0,\tau(x_2,w_2))$, 
 \item[(ii)] the  length of $\gamma$ and $\eta$ does not exceed $T$,
\item[(iii)] $\gamma$ and $\eta$ form an angle $\theta$ at $x_0$ belonging to the interval $[\theta_0,\pi/2]$, 
\item[(iv)] if $d(\gamma(t), \eta(\tau))  <r$, then $|t-t_0| \leq c_0 d(\gamma(t), \eta(\tau))$ and $|\tau-\tau_0| \leq c_0 d(\gamma(t), \eta(\tau))$.
\end{itemize}
\end{itemize}
Here, $\mathrm{Inj}(M)$ denotes the injectivity radius of $M$, defined as follows. Let $(S,g)$ be a closed extension of $(M,g)$, see \cite[Lemma 3.1.8]{PSU_book}, and we assume that $(S,g)$ is fixed once for all. We write $\mathrm{Inj}(M):= \mathrm{Inj}(S)$, and we have $\mathrm{Inj}(M)>0$ as $S$ is compact, see \cite[Lemma 6.16]{Lee_Riem_mfld_2018}. We denote the Riemannian distance as $d(\cdot, \cdot)$. The angle between $\gamma$ and $\eta$ at the point $x_0$ is denoted by $\theta$, where $0 \leq \theta \leq \pi$, and $\cos \theta = g(\dot{\gamma}(t_0), \dot{\eta}(\tau_0))$. Our next main result is as follows. 
\begin{thm}
\label{thm_main_2}
Let $(M,g)$ be a smooth compact oriented Riemannian manifold of dimension $n\ge 2$ with smooth boundary.    Assume that $(M,g)$ satisfies the condition (H1) with some constants $T>0$,  $0<\theta_0 <\pi/2$,  $0<r<\mathrm{Inj}(M)/2$, and $c_0 > 0$. Let $B>0$.  Assume that $q_1,q_2\in C^{0,\alpha}(M)$ are such that $q_1-q_2\in \mathcal{A}(B)$. Then for any $\delta>0$, there exists a set $J\subset [1,\infty)$, $J=J(M,g,\delta)$, satisfying $|J|\le \delta$, and $\lambda_0=\lambda_0(M,g, \delta, B, T, \theta_0, r,c_0)>0$ such that if for some $\lambda\ge \lambda_0$, $\lambda\notin J$,  and $\lambda$ satisfying the assumption (A), we have 
$\Lambda_{q_1}^\lambda=\Lambda_{q_2}^\lambda$ then $q_1=q_2$ in~$M$. 
\end{thm}

Let us make some comments regarding the geometric assumption (H1). The condition $r < \mathrm{Inj}(M)/2$ ensures that the geodesics $\gamma$ and $\eta$ in assumption (H1) do not have self-intersections for $|t-t_0|<r$ and $|\tau-\tau_0|<r$, respectively. The first three conditions are straightforward. As for condition (iv), it is a quantitative way to state that the geodesics $\gamma$ and $\eta$ intersect only at $x_0$ and they do not have self-intersections at $x_0$. Proposition \ref{prop:sec_curvatures} in Appendix \ref{sec_discussion_examples} shows that condition (iv) locally follows from condition (iii).
\begin{exmp}
\label{example_simple_manifold_H1}
If the manifold $(M,g)$ is simple then the geometric assumption (H1) holds, see Appendix \ref{sec_discussion_examples} for the details. 
\end{exmp}

\begin{exmp}
\label{example_cylinder_H1}
The condition (H1) holds for the cylinder  $M = \mathbb{S}^1 \times [0,a]$, $a > 0$, with its usual flat metric, see Appendix \ref{sec_discussion_examples} for the details.
\end{exmp}

Let us next make some comments about the class of potentials in the cubic nonlinearity considered in Theorem \ref{thm_main_2}. The assumption that  $p\in \mathcal{A}(B)$ is similar to the assumption that the perturbation is angularly controlled in \cite[Theorem 2]{Rakesh_Uhlmann_2014} or horizontally controlled in 
\cite{Rakesh_Salo_2020}.   In particular, the assumption $p\in \mathcal{A}(B)$ is always satisfied for some $B$ if $p$ lies in a finite-dimensional space. We refer to Appendix \ref{sec_example} for an example of an infinite set of linearly independent admissible perturbations.

Let us proceed to discuss the main ideas in the proof of Theorem \ref{thm_main} and Theorem \ref{thm_main_2}. The crucial ingredient in both proofs is the polynomial resolvent estimate for the Laplacian on a smooth compact Riemannian manifold without boundary, valid for most frequencies, established in Theorem \ref{thm_resolvent} below. This resolvent estimate could also be of independent interest. Its proof is quite simple and uses only the self-adjoint resolvent bound, see \eqref{eq_2_3}, and a rough version of the Weyl law, see \eqref{eq_2_6}. We refer to \cite{Lafontaine_Spence_Wunsch_2021}, see also \cite{Tang_Zworski_1998, Stefanov_2001}, for polynomial resolvent bounds for most frequencies in the black box scattering.

Once the polynomial resolvent estimates have been established, the proof of Theorem \ref{thm_main} proceeds as follows. First, using the third order linearization of the problem \eqref{eq_1_1}, we derive the integral identity \eqref{eq_3_4} valid for four solutions of the Helmholtz equation $(-\Delta_g-\lambda^2)v=0$ in $M^{\text{int}}$. The construction of such solutions is based on Gaussian beam quasimodes to the Helmholtz equation. To convert these approximate solutions into exact solutions, we rely on the solvability result which is a consequence of the polynomial resolvent estimate of Theorem \ref{thm_resolvent}. Substituting the constructed solutions into the integral identity \eqref{eq_3_4} and using a rough stationary phase argument combined with the boundary determination of the potential, allows us to complete the proof of Theorem \ref{thm_main}.

The proof of Theorem \ref{thm_main_2} follows a similar approach to that of Theorem \ref{thm_main}. However, we introduce a crucial new element that allows us to remove the dependence of the frequency $\lambda_0$ on a specific point $x_0$ and the specific choice of two non-tangential geodesics passing through this point. This crucial element involves the invariant construction of Gaussian beam quasimodes, with uniform bounds for the underlying constants, as recently established in \cite[Theorem 6.2]{Ma_Sahoo_Salo_2022}, see also \cite{St-Amant_2024}. Additionally, we conduct a more thorough analysis of all constants involved during the proof.

The paper is organized as follows. Section \ref{sec_resolvent} contains the proof of polynomial resolvent estimates for the Laplacian on a smooth compact Riemannian manifold. The proof of Theorem \ref{thm_main} is given in Section \ref{sec_proof_main}. Section \ref{sec_proof_main_2} is devoted to the proof of Theorem \ref{thm_main_2}. 
The proof of Theorem \ref{thm:intersection} is presented in Section \ref{sec:intersection}.  Appendix \ref{sec_rough_stationary_phase} discusses a standard rough version of the stationary phase lemma required for the proofs of Theorem \ref{thm_main} and Theorem \ref{thm_main_2}.  A boundary determination of a potential from the integral identity  \eqref{eq_3_4}, used in the proofs of Theorem \ref{thm_main} and Theorem \ref{thm_main_2}, is presented in Appendix \ref{app_boundary_determination}. Appendix \ref{sec_example} contains an example of an infinite set of linearly independent admissible perturbations. Finally, Appendix \ref{sec_discussion_examples} provides examples of manifolds that satisfy the assumption (H1).

\section{Polynomial resolvent estimates for the Laplacian on a smooth compact Riemannian manifold}
\label{sec_resolvent}

Let $(N,g)$ be a smooth compact Riemannian manifold of dimension $n\ge 2$ without boundary. Let $-\Delta_g$ be the positive Laplace operator on $N$, associated with the metric $g$.  It is a self-adjoint operator on $L^2(N)$ with the domain $H^2(N)$, the standard Sobolev space on $N$, and it has a discrete spectrum $\text{Spec}(-\Delta_g)\subset [0,\infty)$.  

When $\lambda>0$, $\lambda^2\notin \text{Spec}(-\Delta_g)$, the resolvent $(-\Delta_g-\lambda^2)^{-1}: L^2(N)\to L^2(N)$ is a bounded operator.  Motivated by \cite[Theorem 1.1, Theorem 3.3]{Lafontaine_Spence_Wunsch_2021}, we have the following result establishing a polynomial resolvent estimate for the 
Laplacian on $N$, valid for most frequencies. In what follows we let $|\cdot|$ stand for the Lebesgue measure on $\R$. 

\begin{thm}
\label{thm_resolvent}
Given $\delta>0$ and $\varepsilon>0$, there exist $C=C(N,g, \delta, \varepsilon)>0$ and a set $J=J(N,g, \delta,\varepsilon)\subset [1,\infty)$ with $|J|\le \delta$ such that 
\begin{equation}
\label{eq_2_2}
\|(-\Delta_g-\lambda^2)^{-1}\|_{\mathcal{L}(L^2(N), L^2(N))}\le C\lambda^{n+\varepsilon},
\end{equation}
for all $\lambda\in [1,\infty)\setminus J$.
\end{thm}

\begin{proof}
We shall first consider the resolvent of the semiclassical Laplacian $-h^2\Delta_g$, $0<h\le 1$. When $z \notin \text{Spec}(-h^2\Delta_g)$, by the spectral theorem we have 
\begin{equation}
\label{eq_2_3}
\|(-h^2\Delta_g-z)^{-1}\|_{{\mathcal{L}(L^2(N), L^2(N))}}=\frac{1}{\text{dist}(z,\text{Spec}(-h^2\Delta_g))}.
\end{equation}
We shall first work on the frequency interval $[1,2]$. Let $r(h)>0$ be an arbitrary function of $h\in (0,1]$. Then it follows from \eqref{eq_2_3} that 
\[
\|(-h^2\Delta_g-z)^{-1}\|_{\mathcal{L}(L^2(N), L^2(N))}\le \frac{1}{r(h)},
\]
for all $z\in [1,2]\setminus J(h)$, where 
\[
J(h)=\{z\in [1,2]: \text{dist}(z,\text{Spec}(-h^2\Delta_g))<r(h)\}. 
\]
To bound the measure of the set $J(h)$, we observe that 
\begin{equation}
\label{eq_2_5}
J(h)\subset \bigcup_{z_j(h)\in \text{Spec}(-h^2\Delta_g)\cap [1,2]} (z_j(h)-r(h), z_j(h)+r(h)),
\end{equation}
where $z_j(h)$, $j=1,2,\dots$, stand for the eigenvalues of $-h^2\Delta_g$. It follows from the Weyl law that 
the number of the eigenvalues of $-h^2\Delta_g$ on the interval $[1,2]$, counting with multiplicities,  is  given by 
\begin{equation}
\label{eq_2_6}
\#(\text{Spec}(-h^2\Delta_g)\cap [1,2])=\mathcal{O}(h^{-n}),
\end{equation}
for all $0<h\le 1$, see \cite[Theorem 14.11]{Zworski_book}. Thus, it follows from \eqref{eq_2_5}, \eqref{eq_2_6} that
\[
|J(h)|\le  \sum_{z_j(h)\in \text{Spec}(-h^2\Delta_g)\cap [1,2]} 2r(h) \le \mathcal{O}(h^{-n})r(h),
\]
for all $0<h\le 1$. Letting $\delta'>0$ and $\varepsilon'>0$ be fixed to be  chosen later, we set 
\[
r(h)=\delta' h^{n+\varepsilon'}
\]
 so that 
\begin{equation}
\label{eq_2_7}
 |J(h)|\le \delta' \mathcal{O}(h^{\varepsilon'})\to 0,
\end{equation}
 as $h\to 0$.  Summarizing the discussion so far, we have shown that for all $0<h\le 1$, there exists a subset $J(h)\subset [1,2]$ satisfying \eqref{eq_2_7} such that  
\begin{equation}
\label{eq_2_8}
\|(-h^2\Delta_g-z)^{-1}\|_{{\mathcal{L}(L^2(N), L^2(N))}}\le (\delta')^{-1} h^{-n-\varepsilon'},
\end{equation}
 for all $z\in [1,2]\setminus J(h)$. 
 
Next, we shall obtain the bound \eqref{eq_2_2} for the nonsemiclassical resolvent
\[
(-\Delta_g-\lambda^2)^{-1}=h^2(-h^2\Delta_g-h^2\lambda^2)^{-1},
\]
for $\lambda\in [1,\infty)$, outside some set of small measure, containing the spectrum of $\sqrt{-\Delta_g}$. In doing so, we shall follow the proof of \cite[Theorem 3.3]{Lafontaine_Spence_Wunsch_2021}, and \cite[page 767]{Sjostrand_Zworski_1991}.   We write 
\[
[1,\infty)=\bigcup_{l=0}^\infty[2^l, 2^{l+1}).
\]
Now if $\lambda\in [1,\infty)$ then $\lambda^2\in [2^l, 2^{l+1})$ for some unique $l=0,1,2,\dots$, and therefore, $2^{-l}\lambda^2\in [1,2)\subset[1,2]$. Letting $h=2^{-l/2}\in (0,1]$, and using \eqref{eq_2_8}, we obtain that 
\begin{equation}
\label{eq_2_9}
\begin{aligned}
\|(-\Delta_g-&\lambda^2)^{-1}\|_{\mathcal{L}(L^2(N), L^2(N))}=h^2\|(-h^2\Delta_g-h^2\lambda^2)^{-1}\|_{\mathcal{L}(L^2(N), L^2(N))}\\
&\le (\delta')^{-1} h^{2-n-\varepsilon'}
\le (\delta')^{-1} 2^{l/2(n+\varepsilon'-2)}\le (\delta')^{-1} \lambda^{n-2+\varepsilon'},
\end{aligned}
\end{equation}
 for all $\lambda^2\in [2^l,2^{l+1})\setminus\tilde J_l$. Here  we let $\tilde J_l:= 2^lJ(2^{-l/2})\cap [2^l,2^{l+1})$.
We set 
\[
\tilde J:=\bigcup_{l=0}^\infty \tilde J_l
\]
so that the bound \eqref{eq_2_9} holds for all $\lambda^2\in [1,\infty)\setminus \tilde J$. Letting $\varepsilon>0$ be arbitrary fixed, choosing 
\begin{equation}
\label{eq_2_10}
\varepsilon'=2+\varepsilon,
\end{equation}
and using \eqref{eq_2_7},
 we get 
\begin{equation}
\label{eq_2_11}
|\tilde J|\le \sum_{l=0}^\infty 2^l |J(2^{-l/2})|\le \mathcal{O}(\delta')\sum_{l=0}^\infty \bigg(\frac{1}{2^{\varepsilon'/2-1}}\bigg)^l=\mathcal{O}(\delta')\sum_{l=0}^\infty \bigg(\frac{1}{2^{\varepsilon/2}}\bigg)^l\le \mathcal{O}_{\varepsilon}(\delta').
\end{equation}
Letting $\delta>0$ be arbitrary fixed, and choosing $\delta'=\delta'(\delta, \varepsilon)>0$ so that  $ \mathcal{O}_{\varepsilon}(\delta')\le 2\delta$, we get  from \eqref{eq_2_11} that 
\begin{equation}
\label{eq_2_12}
|\tilde J|\le 2\delta.
\end{equation}
It follows from \eqref{eq_2_9}, \eqref{eq_2_10}, and \eqref{eq_2_12} that 
\[
\|(-\Delta_g-\lambda^2)^{-1}\|_{L^2(N)\to L^2(N)}\le C\lambda^{n+\varepsilon},
\]
for all $\lambda^2\in [1,\infty)\setminus \tilde J$, where $C=C(\delta, \varepsilon)>0$ and $\tilde J$ satisfies \eqref{eq_2_12}. We define the set $J\subset [1,\infty)$ so that $\lambda\in [1,\infty)\setminus J$ if and only if $\lambda^2\in [1,\infty)\setminus \tilde J$.  Letting $\chi_J$ and $\chi_{\tilde J}$ be the characteristic functions of $J$ and $\tilde J$, respectively, we get 
\[
|\tilde J|=\int_{1}^\infty \chi_{\tilde J}(y)dy=\int_1^\infty \chi_{\tilde J}(x^2)2xdx=\int_1^\infty\chi_J(x)2xdx\ge 2|J|,
\]
and therefore, $|J|\le  |\tilde J|/2\le \delta$. 
\end{proof}

\begin{rem}
A result similar to Theorem \ref{thm_resolvent} is valid when $-\Delta_g$ is replaced by the Schr\"odinger operator $-\Delta_g+q$ with $q\in L^\infty(N;\R)$. We refer to \cite[Theorem 10.1, Remark 10.2]{Markus_book_1988} for the Weyl law for $-\Delta_g+q$.
\end{rem}

Let $H^s(N)$, $s\in\R$, be the standard Sobolev space on $N$, equipped with the norm
\[
\|u\|_{H^s(N)}=\|(1-\Delta_g)^{s/2}u\|_{L^2(N)}, 
\] 
where the Bessel potential $(1-\Delta_g)^{s/2}$ is defined by the self-adjoint functional calculus. Note that 
\begin{equation}
\label{eq_2_12_11}
\|u\|_{H^1(N)}^2=((1-\Delta_g)u,u)_{L^2(N)}=\|\nabla_g u\|_{L^2(N)}^2+\|u\|_{L^2(N)}^2.
\end{equation}
We have the following consequence of Theorem \ref{thm_resolvent}.
\begin{cor}
\label{cor_resolvent}
Given $\delta>0$ and $\varepsilon>0$, there exists $C=C(N,g, \delta, \varepsilon)>0$ and a set $J=J(N,g, \delta,\varepsilon)\subset [1,\infty)$ with $|J|\le \delta$ such that 
\[
\|(-\Delta_g-\lambda^2)^{-1}\|_{L^2(N)\to H^1(N)}\le C\lambda^{1+n+\varepsilon},
\]
for all $\lambda\in [1,\infty)\setminus J$.
\end{cor}
\begin{proof}
The result follows by multiplying the equation $-\Delta_g u-\lambda^2 u=f$ by $\overline{u}$,  integrating by parts, using \eqref{eq_2_12_11}, and Theorem \ref{thm_resolvent}. 
\end{proof}

Let $(M,g)$ be a smooth compact Riemannian manifold of dimension $n\ge 2$ with smooth boundary and let $\lambda\ge 1$. In the proofs of Theorem \ref{thm_main} and Theorem \ref{thm_main_2}, we shall construct complex geometric optics solutions to the Helmholtz equation  $-\Delta_g u-\lambda^2 u=0$ in $M^{\text{int}}$. In doing so, we shall need the following consequence of the resolvent estimates of Theorem \ref{thm_resolvent}. To state the result, for convenience, we use the semiclassical notation and write $h=\lambda^{-1}$. Assume, as we may, that $(M,g)$ is embedded in a compact smooth Riemannian manifold $(N,g)$ without boundary of the same dimension. Let $U\subset N$ be open such that $M\subset U$. 
Let $H^s(N)$, $s\in \R$, be the standard Sobolev space on $N$, equipped with the natural semiclassical norm, 
\begin{equation}
\label{eq_4_2}
\|u\|_{H^s_{\text{scl}}(N)}=\|J^s u\|_{L^2(N)},
\end{equation}
where the Bessel potential $J^s=(1-h^2\Delta_g)^{s/2}$ is defined by the self-adjoint functional calculus on $L^2(N)$. 
\begin{prop}
Let $s\in \R$. 
Given $\delta>0$ and $\varepsilon>0$, there exists $C=C(N,g, \delta, \varepsilon)>0$ and a set $J=J(N,g, \delta,\varepsilon)\subset [1,\infty)$ with $|J|\le \delta$ such that 
\begin{equation}
\label{eq_4_3}
\|u\|_{H^s_{\emph{\text{scl}}}(N)}\le  C h^{-2-n-\varepsilon} \|(-h^2\Delta_g-1)u\|_{H^s_{\emph{\text{scl}}}(N)}, \quad u\in C^\infty_0(M^{\text{int}}),
\end{equation}
for all $h>0$ small enough such that $h^{-1}\in [1,\infty)\setminus J$. 
\end{prop}

\begin{proof}
We shall follow \cite[Lemma 4.3]{DKSaloU_2009}, see also \cite[Proposition 2.2]{Krup_Uhlmann_nonlinear_mag_2018}, \cite[Lemma 3.2]{Ma_Sahoo_Salo_2022}. First it follows from Theorem \ref{thm_resolvent} that 
\begin{equation}
\label{eq_4_1}
\|u\|_{L^2(N)}\le  C h^{-2-n-\varepsilon}\|(-h^2\Delta_g-1)u\|_{L^2(N)}, \quad u\in C^\infty_0(U),
\end{equation}
for all $h>0$ such that $h^{-1}\in [1,\infty)\setminus J$. 

Now note that $J^s$ is a semiclassical pseudodifferential operator of order $s$ on $N$, i.e. $J^s\in \text{Op}_h(S^s(T^*N))$, see \cite[Proposition 16.1.1]{Sjostrand_book_2019}, and see \cite[Section 2]{Krup_Uhlmann_nonlinear_mag_2018} for the definition of the symbol class $S^s(T^*N)$ and  the standard $h$-quantization $\text{Op}_h(S^s(T^*N))$. Therefore, we have the following pseudolocal estimates: if $\psi,\chi\in C^\infty(N)$ with $\chi=1$ near $\text{supp}(\psi)$ and if $\alpha, \beta\in\R$ then 
\begin{equation}
\label{eq_4_4}
(1-\chi)J^s\psi=\mathcal{O}(h^\infty): H^\alpha_{\text{scl}}(N)\to H^\beta_{\text{scl}}(N).
\end{equation}
 all $0<h\le 1$. 
Let $\chi\in C_0^\infty(U)$ be such that $\chi=1$ near $M$. Then using \eqref{eq_4_1}, \eqref{eq_4_2},  and \eqref{eq_4_4}, we get for for all $h>0$ such that $h^{-1}\in [1,\infty)\setminus J$ and $u\in C^\infty_0(M^{\text{int}})$ that 
\begin{equation}
\label{eq_4_5}
\begin{aligned}
\|u\|_{H^s_{\text{scl}}(N)}&\le \|\chi J^s u\|_{L^2(N)}+\|(1-\chi) J^s u\|_{L^2(N)}\\
&\le C h^{-2-n-\varepsilon}\|(-h^2\Delta_g-1)(\chi J^s u)\|_{L^2(N)}+ \mathcal{O}(h^\infty)\|u\|_{H^s_{\text{scl}}(N)}\\
&\le C h^{-2-n-\varepsilon}\big(\|\chi  (-h^2\Delta_g-1)(J^s u)\|_{L^2(N)}+ \| [-h^2\Delta_g, \chi ](J^s u)\|_{L^2(N)}\big)\\
&+ \mathcal{O}(h^\infty)\|u\|_{H^s_{\text{scl}}(N)}\\
&\le C h^{-2-n-\varepsilon} \| (-h^2\Delta_g-1)(J^s u)\|_{L^2(N)}+\mathcal{O}(h^\infty)\|u\|_{H^s_{\text{scl}}(N)}\\
&\le C h^{-2-n-\varepsilon} \|J^s (-h^2\Delta_g-1) u\|_{L^2(N)} +\mathcal{O}(h^\infty)\|u\|_{H^s_{\text{scl}}(N)}\\
&\le C h^{-2-n-\varepsilon} \| (-h^2\Delta_g-1) u\|_{H^s_{\text{scl}}(N)} +\mathcal{O}(h^\infty)\|u\|_{H^s_{\text{scl}}(N)}.
\end{aligned}
\end{equation}
Here we have used that 
\[
\| [-h^2\Delta_g, \chi ](J^s u)\|_{L^2(N)}\le \mathcal{O}(h^\infty)\|u\|_{H^s_{\text{scl}}(N)},
\]
which is a consequence of  \eqref{eq_4_4}, and the fact that $[-h^2\Delta_g, J^s]=0$. Absorbing the error term $\mathcal{O}(h^\infty)\|u\|_{H^s_{\text{scl}}(N)}$ in the left hand side of \eqref{eq_4_5}, we get the estimate \eqref{eq_4_3} for $h$ small enough such that $h^{-1}\in [1,\infty)\setminus J$. 
\end{proof}

Using a standard argument, see \cite{DKSaloU_2009}, we convert the a priori estimate \eqref{eq_4_3}  into the following solvability result. 
\begin{prop}
\label{prop_solvability}
Let $s\in \R$, $\delta>0$ and $\varepsilon>0$. Then if $h>0$ is small enough such that $h^{-1}\in [1,\infty)\setminus J$, $|J|\le \delta$, then for any $v\in H^s(M^{\text{int}})$, there is a solution $u\in H^s(M^{\text{int}})$ of the equation
\[
(-h^2\Delta_g-1)u=v\quad \text{in}\quad M^{\text{int}},
\]
which satisfies 
\[
\|u\|_{H^s_{\text{scl}}(M^{\text{int}})}\le C h^{-2-n-\varepsilon}\|v\|_{H^s_{\text{scl}}(M^{\text{int}})}, \quad C=C(M,g,\delta, \varepsilon)>0.
\]
\end{prop}
Here 
\[
H^s(M^{\text{int}})=\{V|_{M^{\text{int}}}: V\in H^s(N)\}, \quad s\in \R, 
\]
with the norm 
\[
\|v\|_{H^s_{\text{scl}}(M^{\text{int}})}=\inf_{V\in H^s(N),\, v=V|_{M^{\text{int}}}}\|V\|_{H^s_{\text{scl}}(N)}.
\]

\section{Proof of Theorem \ref{thm_main}}
\label{sec_proof_main}

\subsection{Gaussian beam quasimodes and solutions to Helmholtz equations}


Let $(M,g)$ be a smooth compact oriented Riemannian manifold of dimension $n\ge 2$ with smooth boundary. Let $\gamma: [0, T]\to M$, $0<T<\infty$, be a unit speed non-tangential geodesic on $M$. Here following \cite{DKurylevLS_2016}, we say that a unit speed geodesic $\gamma: [0, T]\to M$, $0<T<\infty$, is non-tangential if $\gamma(0), \gamma(T)\in \p M$, $\gamma(t)\in M^{\text{int}}$ for all $0<t<T$, and $\dot{\gamma}(0),\dot{\gamma}(T)$ are non-tangential vectors on $\p M$. 

We shall need the following well-known result concerning the construction of Gaussian beam quasimodes on $M$, localized to the non-tangential geodesic $\gamma$, see \cite{DKurylevLS_2016, Feizmohammadi_Oksanen_2020, LLLS_2021}. As in \cite{LLLS_2021}, it will be convenient to normalize our quasimodes in $L^4(M)$, as later we have to work with products of four such quasimodes. 

\begin{prop}
\label{prop_Gaussian_quasimodes}
For any $k\in \N\cup\{0\}$ and $R\ge 1$, there exist $N\in \N$ and a family of Gaussian beam quasimodes $v(\cdot; \lambda)\in C^\infty(M)$, $\lambda\ge 1$,  such that $\mathrm{supp}(v(\cdot;\lambda))$ is confined to a small neighborhood of $\gamma([0,T])$ and 
\begin{equation}
\label{eq_4_0_-1}
\|(-\Delta_g-\lambda^2)v(\cdot; \lambda)\|_{H^k(M^{\text{int}})}=\mathcal{O}(\lambda^{-R}),
\end{equation}
\[
\|v(\cdot;\lambda)\|_{L^4(M)}=\mathcal{O}(1), \quad \|v(\cdot;\lambda)\|_{L^\infty(M)}=\mathcal{O}(\lambda^{\frac{n-1}{8}}),
\]
as $\lambda\to \infty$. The local structure of the family $v(\cdot; \lambda)$ is as follows: let $p\in \gamma([0,T])$ and let $t_1<\cdots<t_{N_p}$ be the times in $[0, T]$ when $\gamma(t_l)=p$, $l=1,\dots, N_p$. In a sufficiently small neighborhood $U$ of a point $p$, we have
\[
v|_{U}=v^{(1)}+\cdots+v^{(N_p)},
\]
where each $v^{(l)}$ has the form 
\[
v^{(l)}(x;\lambda)=\lambda^{\frac{n-1}{8}} e^{i\lambda \varphi^{(l)}(x)}a^{(l)}(x;\lambda).
\]
Here $\varphi=\varphi^{(l)}\in C^\infty(\overline{U};\C)$ satisfies for $t$ near $t_l$,
\[
\varphi(\gamma(t))=t, \quad \nabla\varphi(\gamma(t))=\dot{\gamma}(t), \quad \emph{\text{Im}}(\nabla^2 \varphi(\gamma(t)))\ge 0, \quad \emph{\text{Im}}(\nabla^2\varphi)|_{\dot{\gamma}(t)^\perp}>0,
\]
and $a^{(l)}\in C^\infty(\overline{U})$ are of the form,
\[
a^{(l)}(\cdot; \lambda)=\bigg(\sum_{j=0}^N \lambda^{-j}a_j^{(l)}\bigg)\chi\bigg(\frac{y}{\delta'}\bigg),
\]
where $a^{(l)}_0(t,y)=a^{(l)}_{00}(t)+\mathcal{O}(|y|)$ with $a^{(l)}_{00}(t)\ne 0$ for all $t$. Here $(t,y)$ are the Fermi coordinates for $\gamma$ when $t$ near $t_l$, $\chi\in C^\infty_0(\R^{n-1})$ is such that $0\le \chi\le 1$, $\chi=1$ for $|y|\le 1/4$ and $\chi=0$ for $|y|\ge 1/2$, and $\delta'>0$ is a fixed number that can be taken arbitrarily small. 
\end{prop}

Let $\lambda\ge 1$. Next, we shall construct complex geometric optics solutions to the Helmholtz equation 
\begin{equation}
\label{eq_4_0}
-\Delta_g u-\lambda^2 u=0\quad \text{in}\quad M^{\text{int}},
\end{equation}
based on the Gaussian beam quasimodes of Proposition \ref{prop_Gaussian_quasimodes}.  Specifically, we look for solutions to \eqref{eq_4_0} in the form, 
\[
u(\cdot;\lambda)=v(\cdot;\lambda)+r(\cdot;\lambda),
\]
where $v(\cdot;\lambda)$ is the Gaussian beam quasimode given in  Proposition \ref{prop_Gaussian_quasimodes}, and $r(\cdot;\lambda)$ is the remainder term. Hence, $u$ solves \eqref{eq_4_0} provided that $r$ satisfies the equation, 
\begin{equation}
\label{eq_4_6}
(-\Delta_g -\lambda^2)r=-(-\Delta_g -\lambda^2)v, \quad \text{in}\quad M^{\text{int}}.
\end{equation}
Let  $k\in \N$ and $R\ge 1$ be large. Letting  $h=\lambda^{-1}$,  it follows from Proposition \ref{prop_solvability} that for all $h>0$ small enough such that $h^{-1}\in [1,\infty)\setminus J$, $|J|\le \delta$, there is $r\in H^k(M^{\text{int}})$ satisfying \eqref{eq_4_6} such that 
\begin{equation}
\label{eq_4_7}
\|r\|_{H^k_{\text{scl}}(M^{\text{int}})}\le \mathcal{O}(h^{-2-n-\varepsilon})\|(-h^2\Delta_g-1)v\|_{H^k_{\text{scl}}(M^{\text{int}})}\le \mathcal{O}(h^{R-n-\varepsilon}).
\end{equation}
Here we have used \eqref{eq_4_0_-1}. 
Thus, for any $K$ and $k$, there is $R$ large enough such that \eqref{eq_4_7} gives that 
\[
\|r\|_{H^k(M^{\text{int}})}\le h^{-k}\|r\|_{H^k_{\text{scl}}(M^{\text{int}})}=\mathcal{O}(h^K).
\]
We summarize the discussion above in the following proposition.
\begin{prop}
\label{prop_solutions}
Let $k\in \N$,  $K\ge 1$, and $\delta>0$. Then for $\lambda\ge 1$ large enough such that $\lambda\in [1,\infty)\setminus J$, $|J|\le \delta$, there is $u=u(\cdot; \lambda)\in H^k(M^{\emph{\text{int}}})$ solving $(-\Delta_g-\lambda^2)u=0$ in $M^{\emph{\text{int}}}$ and having the form
\[
u(\cdot;\lambda)=v(\cdot;\lambda)+r(\cdot;\lambda),
\]
where $v(\cdot; \lambda)\in C^\infty(M)$ is the Gaussian beam quasimode given in  Proposition \ref{prop_Gaussian_quasimodes} and $r(\cdot;\lambda)\in H^k(M^{\text{int}})$ such that $\|r\|_{H^k(M^{\text{int}})}=\mathcal{O}(\lambda^{-K})$, as $\lambda\to \infty$.
\end{prop}

\begin{rem}
\label{rem_solutions}
Taking $k>n/2+3$ and using the Sobolev embedding $H^k(M^{\text{int}})\subset C^3(M)$, we conclude that $u\in C^3(M)$ with 
\[
\|r\|_{C^3(M)}=\mathcal{O}(\lambda^{-K}), 
\]
as $\lambda\to \infty$.
\end{rem}

\subsection{Main part of the proof of Theorem \ref{thm_main}}

Following \cite{Feizmohammadi_Oksanen_2020, LLLS_2021, Krup_Uhlmann_nonlinear_mag}, we start by performing the third order linearization of the problem \eqref{eq_1_1}  and the Dirichlet--to--Neumann map. To that end letting $\varepsilon=(\varepsilon_1,\varepsilon_2, \varepsilon_3)\in \C^3$ and $f_k\in C^{2,\alpha}(\p M)$, $k=1,2,3$, we  consider the Dirichlet problem
\begin{equation}
\label{eq_3_1}
\begin{cases}
-\Delta_g u_j-\lambda^2u_j+q_{j}(x)u_j^3=0 &\text{in}\quad M^{\text{int}},\\
u_j=\varepsilon_1f_1+\varepsilon_2f_2+ \varepsilon_3f_3& \text{on}\quad \p M,
\end{cases}
\end{equation} 
$j=1,2$. It follows from \cite[Theorem B.1]{Krup_Uhlmann_nonlinear_mag},  see also \cite[Proposition 1]{Feizmohammadi_Oksanen_2020} and \cite[Proposition 2.1]{LLLS_2021}, that for all $|\varepsilon|$ sufficiently small, the problem \eqref{eq_3_1} has a unique small solution $u_j=u_j(x,\varepsilon)$, which depends holomorphically on $\varepsilon\in \text{neigh}(0, \C^3)$. Differentiating \eqref{eq_3_1} with respect to $\varepsilon_l$, $l=1,2,3$, and using that $u_j(x,0)=0$, we get 
\begin{equation}
\label{eq_3_2}
\begin{cases}
-\Delta_g v_j^{(l)}-\lambda^2v_j^{(l)}=0 &\text{in}\quad M^{\text{int}},\\
v_j^{(l)}=f_l\quad \p M,
\end{cases}
\end{equation} 
where $v_j^{(l)}=\p_\varepsilon u_j|_{\varepsilon=0}$. By the uniqueness and elliptic regularity for \eqref{eq_3_2}, we have $v^{(l)}:=v_1^{(l)}=v^{(2)}_2\in C^{2,\alpha}(M)$, $l=1,2,3$, see  \cite[Theorem 6.15]{Gilbarg_Trudinger_book}. Applying $\p_{\varepsilon_1}\p_{\varepsilon_2}\p_{\varepsilon_3}|_{\varepsilon=0}$ to \eqref{eq_3_1}, we get the third order linearization, 
\begin{equation}
\label{eq_3_3}
\begin{cases}
-\Delta_g w_j-\lambda^2w_j+6q_{j}(x)v^{(1)}v^{(2)}v^{(3)}=0 &\text{in}\quad M^{\text{int}},\\
w_j=0& \text{on}\quad \p M,
\end{cases}
\end{equation} 
where $w_j:=\p_{\varepsilon_1}\p_{\varepsilon_2}\p_{\varepsilon_3}u_j|_{\varepsilon=0}$. 

The fact that for some $\lambda\in [1,\infty)\setminus J$ sufficiently large, we have $\Lambda_{q_1}^{\lambda}(\varepsilon_1f_1+\varepsilon_2f_2+ \varepsilon_3f_3)=\Lambda_{q_2}^{\lambda}(\varepsilon_1f_1+\varepsilon_2f_2+ \varepsilon_3f_3)$ for all small $\varepsilon$ and $f_j\in C^{2,\alpha}(\p M)$ yields that $\p_\nu u_1|_{\p M}=\p_\nu u_2|_{\p M}$. Therefore, $\p_\nu w_1|_{\p M}=\p_\nu w_2|_{\p M}$. Multiplying the difference of two equations in \eqref{eq_3_3} by $v^{(4)}\in C^{2,\alpha}(M)$ such that $-\Delta_g v^{(4)}-\lambda^2v^{(4)}=0$ in $M^{\text{int}}$, and using Green's formula, we get 
\begin{equation}
\label{eq_3_4}
\int_{M} p v^{(1)}v^{(2)}v^{(3)}v^{(4)}dV_g=0.
\end{equation}
Here $p=q_1-q_2$. Note that \eqref{eq_3_4} holds for all $v^{(l)}\in C^{2,\alpha}(M)$ such that $-\Delta_g v^{(l)}-\lambda^2v^{(l)}=0$ in $M^{\text{int}}$, $l=1,\dots, 4$. 

We shall next show that $p=0$. To that end,  let $x_0\in M$ be such that 
\begin{equation}
\label{eq_3_4_sup}
\sup_{x\in M}|p(x)|=|p(x_0)|.
\end{equation} 
Assume first that $x_0\in M^{\text{int}}$. By the condition (H), there are two non-tangential unit speed geodesics $\gamma:[0, T]\to M$ and $\eta:[0,S ]\to M$, $0<T, S<\infty$, such that $\gamma$ and $\eta$ do not have self-intersections at $x_0$ and $x_0$ is the only point of their intersections.

Furthermore, there is a constant $B > 0$ such that $p \in \mathcal{A}(B)$. Throughout the following discussion, it is important to note that all implicit constants might rely on several factors, such as the manifold $(M,g)$, the constant $B$ ensuring $p \in \mathcal{A}(B)$, the point $x_0$, and the specific geodesics $\gamma$ and $\eta$, including their lengths and the angle between them. Since in Theorem \ref{thm_main}, we simply assert that the frequency $\lambda_0$ is dependent on $(M,g)$ and $p$, without specifying the precise nature of this dependence, to avoid clutter, we will not explicitly state the dependence on these factors in the implicit constants mentioned below.

Using Proposition \ref{prop_solutions} and Remark \ref{rem_solutions},  for $\lambda\in [1,\infty)\setminus J$  large enough,  we have $v^{(l)} \in C^3(M)$ solving $-\Delta_g v^{(l)}-\lambda^2v^{(l)}=0$ in $M^{\text{int}}$, $l=1,2$, of the form 
\begin{equation}
\label{eq_3_5}
v^{(1)}=v+r_1, \quad v^{(2)}=w+r_2,
\end{equation}
where $v=v(\cdot;\lambda), w=w(\cdot;\lambda)\in C^\infty(M)$ are Gaussian beam quasimodes concentrating near the geodesics $\gamma$ and $\eta$, respectively, given in Proposition \ref{prop_Gaussian_quasimodes}, and 
\begin{equation}
\label{eq_3_6}
\|r_l\|_{L^\infty(M)}=\mathcal{O}(\lambda^{-K}),
\end{equation}
as $\lambda\to \infty$, $K\gg 1$. We set 
\begin{equation}
\label{eq_3_7}
v^{(3)}=\overline{v^{(1)}}, \quad v^{(4)}=\overline{v^{(2)}}.
\end{equation}
It follows from Proposition \ref{prop_Gaussian_quasimodes} that in a sufficiently small neighborhood of the point $x_0$, we have
\begin{equation}
\label{eq_3_8}
v(x;\lambda)=\lambda^{\frac{n-1}{8}} e^{i\lambda \varphi(x)}a(x;\lambda), \quad w(x;\lambda)=\lambda^{\frac{n-1}{8}} e^{i\lambda \psi(x)}b(x;\lambda), 
\end{equation}
where 
\begin{equation}
\label{eq_3_9}
\begin{aligned}
&\varphi(\gamma(t))=t, \quad \nabla\varphi(\gamma(t))=\dot{\gamma}(t), \quad \text{Im}(\nabla^2 \varphi(\gamma(t)))\ge 0, \quad \text{Im}(\nabla^2\varphi)|_{\dot{\gamma}(t)^\perp}>0,\\
&\psi(\eta(\tau))=\tau, \quad \nabla\psi(\eta(\tau))=\dot{\eta}(\tau), \quad \text{Im}(\nabla^2 \psi(\eta(\tau)))\ge 0, \quad \text{Im}(\nabla^2\psi)|_{\dot{\eta}(\tau)^\perp}>0,
\end{aligned}
\end{equation}
and 
\begin{equation}
\label{eq_3_10}
\begin{aligned}
&a(t,y; \lambda)=\sum_{j=0}^N \lambda^{-j}\tilde a_j(t,y), \quad \tilde a_j(t,y)=a_j(t,y) \chi\bigg(\frac{y}{\delta'}\bigg),\\ &b(\tau,z; \lambda)=\sum_{j=0}^N \lambda^{-j}\tilde b_j(\tau,z), \quad \tilde b_j(\tau,z) =b_j(\tau,z) \chi\bigg(\frac{z}{\delta'}\bigg).
\end{aligned}
\end{equation}
Here 
\begin{align*}
&a_0(t,y)=a_{00}(t)+\mathcal{O}(|y|), \quad a_{00}(t)\ne 0, \quad \forall t, \\
&b_0(\tau,z)=b_{00}(\tau)+\mathcal{O}(|z|), \quad b_{00}(\tau)\ne 0, \quad \forall \tau, 
\end{align*}
 $(t,y)$ and $(\tau, z)$ are the Fermi coordinates for $\gamma$ and $\eta$, $\chi\in C^\infty_0(\R^{n-1})$ is such that $0\le \chi\le 1$, $\chi=1$ for $|y|\le 1/4$ and $\chi=0$ for $|y|\ge 1/2$, and $\delta'>0$ is a fixed number that can be taken arbitrarily small. We also have
\begin{equation}
\label{eq_3_11}
\begin{aligned}
&\|v\|_{L^4(M)}=\mathcal{O}(1), \quad \|w\|_{L^4(M)}=\mathcal{O}(1), \\
& \|v\|_{L^\infty(M)}=\mathcal{O}(\lambda^{\frac{n-1}{8}}), \quad \|w\|_{L^\infty(M)}=\mathcal{O}(\lambda^{\frac{n-1}{8}}), 
 \end{aligned}
\end{equation}
as $\lambda\to \infty$. 
Now  thanks to the bounds \eqref{eq_3_6},  \eqref{eq_3_11},  for any $L>0$, we can find $K$ large so that we have 
\begin{equation}
\label{eq_3_12}
v^{(1)}v^{(2)}v^{(3)}v^{(4)}=|v|^2|w|^2+R,
\end{equation}
where 
\begin{equation}
\label{eq_3_13}
\|R\|_{L^\infty(M)}=\mathcal{O}(\lambda^{-L}).
\end{equation}
  Substituting \eqref{eq_3_12} into the integral identity \eqref{eq_3_4}, and using \eqref{eq_3_13}, we get 
\begin{equation}
\label{eq_3_14}
\bigg|\int_M p |v|^2|w|^2 dV_g \bigg|\le \mathcal{O}(\lambda^{-L})\|p\|_{L^\infty(M)}.
\end{equation}

Now recall that $x_0$ is the only intersection point of the geodesics $\gamma$ and $\eta$. Let $\tilde \delta>0$. Taking $\delta'>0$ in \eqref{eq_3_10} to be sufficiently small but fixed, we could achieve that the support of the product $|v|^2|w|^2$ is contained in an open geodesic ball $B_{\tilde \delta}(x_0)$ in $M$ of radius $\tilde \delta$ centered at $x_0$. In $B_{\tilde \delta}(x_0)$,  using \eqref{eq_3_8} and \eqref{eq_3_10}, we get 
\begin{equation}
\label{eq_3_15}
|v|^2|w|^2=\lambda^{\frac{n-1}{2}}e^{-2\lambda(\text{Im}\, \varphi+\text{Im}\, \psi)}\big(|\tilde a_0|^2|\tilde b_0|^2+\mathcal{O}_{L^\infty(B_{\tilde \delta}(x_0))}\big(\lambda^{-1}\big)\big). 
\end{equation}

Letting 
\begin{equation}
\label{eq_3_16}
\Psi=2(\text{Im}\, \varphi+\text{Im}\, \psi)\ge 0, 
\end{equation}
we have from \eqref{eq_3_16} and \eqref{eq_3_9} that 
\begin{equation}
\label{eq_3_18}
\Psi(x_0)=0, \quad d\Psi(x_0)=0, \quad \nabla^2\Psi(x_0)>0.
\end{equation}
The latter inequality follows from the fact that the Hessians of $\text{Im}\, \varphi$ and $\text{Im}\, \psi$ at $x_0$ are positive semidefinite and positive definite in the directions orthogonal to $\gamma$ and $\eta$, respectively. 
If $z=(z_1,\dots,z_n)$ are normal coordinates centered at $x_0$, we have 
\[
\Psi(z)=\frac{1}{2}\Psi''(0)z\cdot z+\mathcal{O}(|z|^3),
\] 
where the Hessian $\Psi''(0)>0$. Thus, by choosing $\tilde \delta$ small enough, we get 
\begin{equation}
\label{eq_3_18_lower_bound}
\Psi(z)\ge c|z|^2 \quad \text{in}\quad B_{\tilde \delta}(x_0),
\end{equation}
for some $c>0$.

Substituting \eqref{eq_3_15} into \eqref{eq_3_14}, we obtain that
\begin{equation}
\label{eq_3_17}
\begin{aligned}
\bigg|\lambda^{\frac{n-1}{2}} \int_{B_{\tilde \delta}(x_0)} p  e^{-\lambda \Psi} |\tilde a_0|^2|\tilde b_0|^2  dV_g \bigg|\le &\mathcal{O}(\lambda^{-L})\|p\|_{L^\infty(M)}
\\
&+\mathcal{O}(\lambda^{\frac{n-1}{2}-1})\|p\|_{L^\infty(M)}\int_{B_{\tilde \delta}(x_0)} e^{-\lambda \Psi} dV_g.
\end{aligned}
\end{equation}
Writing the integral in the normal coordinates, using  $dV_g (z)=|g(z)|^{1/2}dz$ and \eqref{eq_3_18_lower_bound}, we get
\begin{equation}
\label{eq_3_17_new}
\int_{B_{\tilde \delta}(x_0)} e^{-\lambda \Psi} dV_g\le \mathcal{O}(1)  \int_{\R^n} e^{-\lambda c|z|^2} dz=\mathcal{O}(\lambda^{-n/2}).
\end{equation}
Thus, fixing 
\begin{equation}
\label{eq_3_17_value_L}
L= 3/2,
\end{equation}
 we obtain from \eqref{eq_3_17} and \eqref{eq_3_17_new} that 
\begin{equation}
\label{eq_3_19}
\begin{aligned}
\bigg|\lambda^{\frac{n}{2}} \int_{B_{\tilde \delta}(x_0)} p  e^{-\lambda \Psi} |\tilde a_0|^2|\tilde b_0|^2  dV_g \bigg|\le \mathcal{O}(\lambda^{-1})\|p\|_{L^\infty(M)}.
\end{aligned}
\end{equation}

Next, we shall use the rough stationary phase Lemma \ref{lem_stationary_phase} in the integral in \eqref{eq_3_19}. To that end, we shall write the integral in the normal coordinates $z$ centered at $x_0$, and redefine $\tilde \delta>0$ to be smaller if necessary. We get 
\begin{equation}
\label{eq_3_19_new}
\begin{aligned}
&\lambda^{\frac{n}{2}} \int_{B_{\tilde \delta}(x_0)} p  e^{-\lambda \Psi} |\tilde a_0|^2|\tilde b_0|^2  dV_g= 
\lambda^{\frac{n}{2}} \int_{B_{\tilde\delta}(0)} p(z)  e^{-\lambda \Psi(z)} |\tilde a_0(z)|^2|\tilde b_0(z)|^2 |g(z)|^{1/2} dz \\
&= \frac{(2\pi)^{\frac{n}{2}}}{(\det \nabla^2\Psi(x_0))^{1/2}}|p(x_0)| |a_{0}(x_0)|^2|b_{0}(x_0)|^2+\mathcal{O}(1) \lambda^{-\frac{\alpha}{2}}\|p\|_{C^{0,\alpha}(M)}.
\end{aligned}
\end{equation}
Here we have also used \eqref{eq_app_Holder_ineq} and the fact that $|g(0)|=1$ in the normal coordinates.  Combining \eqref{eq_3_19} and \eqref{eq_3_19_new}, we obtain that 
\begin{equation}
\label{eq_3_20_new}
\frac{1}{(\det \nabla^2\Psi(x_0))^{1/2}}|p(x_0)| |a_{0}(x_0)|^2|b_{0}(x_0)|^2 \le \mathcal{O}(\lambda^{-\alpha/2})\|p\|_{C^{0,\alpha}(M)}. 
\end{equation}
Therefore, using that $ a_{0}(x_0)\ne 0$, and  $b_{0}(x_0)\ne 0$, 
we get from \eqref{eq_3_20_new} that 
\begin{equation}
\label{eq_3_20}
|p(x_0)|\le  \mathcal{O}(\lambda^{-\alpha/2})\|p\|_{C^{0,\alpha}(M)},
\end{equation} 
for $\lambda\in [1,\infty)\setminus J$ large enough, satisfying the assumption (A). 

It follows from \eqref{eq_3_4_sup} and \eqref{eq_3_20} that 
\begin{equation}
\label{eq_3_21}
\|p\|_{L^\infty(M)}\le \mathcal{O}(\lambda^{-\alpha/2})\|p\|_{C^{0,\alpha}(M)}. 
\end{equation}
 Since $p\in \mathcal{A}(B)$, get from \eqref{eq_3_21} that 
\begin{equation}
\label{eq_3_22}
\|p\|_{L^\infty(M)}\le B\mathcal{O}(\lambda^{-\alpha/2}) \|p\|_{L^\infty(M)}. 
\end{equation}
Now choosing $\lambda\in [1,\infty)\setminus J$ large enough, satisfying the assumption (A) and depending on $B$,  we conclude from \eqref{eq_3_22} that $p=0$ in $M$. 

If $x_0\in \p M$ then \eqref{eq_3_4} and Proposition \ref{prop_app_boundary_deter} imply that $p(x_0)=0$, and therefore, \eqref{eq_3_4_sup}  yields that $p=0$ in $M$. This completes the proof of Theorem \ref{thm_main}.

\section{Proof of Theorem \ref{thm_main_2}}
\label{sec_proof_main_2}

\subsection{Gaussian beam quasimodes and solutions to Helmholtz equations with uniform constants}

\label{sec_solutions_Helmholtz_uniform_constants}

Let $(M,g)$ be a smooth compact oriented Riemannian manifold of dimension $n\ge 2$ with smooth boundary, and let $T_xM$ be the tangent space at the point $x\in M$. We write $\langle\cdot,\cdot\rangle= \langle\cdot,\cdot\rangle_g=g(\cdot,\cdot)$ for the $g$--inner product for tangent vectors and $|\cdot|=|\cdot|_g=\sqrt{\langle\cdot,\cdot\rangle_g}$ for the norm of tangent vectors. The unit sphere bundle $SM$ of $M$ is defined as 
\[
SM=\cup_{x\in M}S_x M, \quad S_xM=\{(x,w)\in T_xM: |w|_g=1\}. 
\] 
The boundary of $SM$ is given by $\p(SM)=\{(x,w)\in SM: x\in \p M\}$.  It is the union of the sets of inward and outward pointing vectors, 
\[
\p_{\pm} SM=\{(x,w)\in \p (SM): \pm \langle w, \nu \rangle \le 0\}, 
\]
where $\nu$ is the outward unit normal to the boundary of $M$. 

Let $(x,w)\in \p_+SM$, and let $\gamma=\gamma_{x,w}(t):[0, \tau(x,w)]\to M$ be the maximally extended unit speed geodesic such that $\gamma(0)=x$ and $\dot{\gamma}(0)=w$. 

Let $T>0$ be fixed. Following   \cite[Theorem 6.2]{Ma_Sahoo_Salo_2022}, we allow the manifold $(M,g)$ to have trapped geodesics, i.e. $\tau(x,w)$ maybe $+\infty$ for some $(x,w)$, but we shall only work with $(x,w)\in \mathcal{G}_T$ where 
\[
\mathcal{G}_T=\{(x,w)\in \p_+SM: \tau(x,w)\le T\}. 
\]
Note that $ \tau(x,w)$ is the length of the geodesic $\gamma_{x,w}$.

We will utilize a recently established result, see \cite[Theorem 6.2]{Ma_Sahoo_Salo_2022}, concerning the construction of Gaussian beam quasimodes on $M$, localized to $\gamma_{x,w}$, with uniform bounds over $(x,w)\in \mathcal{G}_T$. 
\begin{thm}
\label{thm_Gaussian_quasimodes_uniform}
Let $T>0$,  $k\in \N\cup \{0\}$, and $R\ge 0$ be given. There are constants $C=C(M,g,T, k,R)>0$, $\tilde C=\tilde C(M,g,T, k,R)>0$, and $N=N(M,g,T, k,R)\in \N$ such that for any $(x,w)\in \mathcal{G}_T$, there exists a family of Gaussian beam quasimodes  $v=v(\cdot;\lambda)=v_{x,w}(\cdot; \lambda)\in C^\infty(M)$, $\lambda\ge 1$, associated to $\gamma=\gamma_{x,w}$, and satisfying 
\begin{equation}
\label{eq_4_0_-1_uniform}
\|(-\Delta_g-\lambda^2)v(\cdot; \lambda)\|_{H^k(M^{\text{int}})}\le C\lambda^{-R},
\end{equation}
\[
 \|v(\cdot;\lambda)\|_{L^4(M)}\le C, \quad \|v(\cdot;\lambda)\|_{L^\infty(M)}\le C\lambda^{\frac{n-1}{8}},
\]
as $\lambda\to \infty$. There is also a symmetric complex $(1,1)$-tensor $H(t)=H_{x,w}(t)$ on $T_{\gamma(t)}M$, depending smoothly on $t\in [0,\tau(x,w)]$ and satisfying 
\[
\emph{\text{Im}}\,(H(t)^\flat)\ge 0,\quad \emph{\text{Im}}\,(H(t)^\flat)|_{\dot{\gamma}(t)^\perp}\ge {\tilde C}^{-1}g.
\]
The local structure of the family $v(\cdot; \lambda)$ is as follows: if $p\in \gamma([0,\tau(x,w)])$ and $t_1<\cdots<t_{N_p}$ are the times in $[0, \tau(x,w)]$ when $\gamma(t_l)=p$, $l=1,\dots, N_p$, then in a small neighborhood $U$ of $p$, we have
\[
v|_{U}=v^{(1)}+\cdots+v^{(N_p)},
\]
where each $v^{(l)}$ has the form 
\[
v^{(l)}(x;\lambda)=\lambda^{\frac{n-1}{8}} e^{i \lambda\varphi^{(l)}(x)}a^{(l)}(x;\lambda).
\]
Here $\varphi=\varphi^{(l)}\in C^\infty(\overline{U};\C)$ satisfies for $t$ near $t_l$,
\begin{equation}
\label{eq_4_0_-1_phase_l}
\varphi(\gamma(t))=t, \quad \nabla\varphi(\gamma(t))=\dot{\gamma}(t), \quad \nabla^2 \varphi(\gamma(t))= H(t)^\flat, \quad  \|\varphi\|_{C^{k}(\overline{U})}\le \tilde C,
\end{equation}
and $a^{(l)}\in C^\infty(\overline{U})$ are of the form,
\begin{equation}
\label{eq_4_0_-1_a_l}
a^{(l)}(\cdot; \lambda)=\bigg(\sum_{j=0}^N \lambda^{-j}a_j^{(l)}\bigg)\rho,
\end{equation}
where $\rho$ is a smooth cutoff function supported near $\gamma$ when $t$ is near $t_l$, and 
\begin{equation}
\label{eq_4_0_-1_a00}
a^{(l)}_{0}(\gamma(t))=\exp\bigg[-\frac{1}{2}\int_0^t \emph{\text{tr}}_g(H(s))ds\bigg]. 
\end{equation}
One also has $\|a_j^{(l)}\|_{C^{k}(\overline{U})}\le \tilde C$,  $l=1,\dots, N_p, \quad j=0,\dots, N$.
\end{thm}

\begin{rem}
\label{rem_choice_rho}
According to the proof presented in \cite[Theorem 6.2]{Ma_Sahoo_Salo_2022}, the cutoff function $\rho=\rho^{(l)}$ used in the amplitude $a^{(l)}$ in \eqref{eq_4_0_-1_a_l} is selected as $\rho(t,y)=\chi(|y|/\delta_1)$, where $\delta_1=\delta_1(M,g,T, k, R)>0$ is small, see \cite[formula (6.6) and discussion before it]{Ma_Sahoo_Salo_2022}. Here, $(t,y)$ are the Fermi coordinates for $\gamma$ when $t$ near  $t_l$, and  the function $\chi$ is fixed and belongs to the space $C^\infty_0(\R)$, satisfying the properties: $0\leq \chi\leq 1$, $\chi(t)=1$ for $|t|\leq 1/2$, and $\chi(t)=0$ for $|t|\geq 2/3$.
\end{rem}

\begin{rem}  While the constant $C$ in Theorem \ref{thm_Gaussian_quasimodes_uniform} depends on $\delta_1$, the constants $\tilde C$ and $N$ are independent of $\delta_1$. The independence of $\tilde C$ of $\delta_1$ will be crucial in the proof of Theorem \ref{thm_main_2} below.
\end{rem}

\begin{rem}
If $(M,g)$ is non-trapping, i.e. $\tau(x,w)<\infty$ for all $(x,w)\in \p_+SM$, then $\mathcal{G}_T=\p_+SM$ for sufficiently large $T$, see  \cite[Remark 6.1]{Ma_Sahoo_Salo_2022}.
\end{rem}

Let  $K\ge 1$. Through a similar line of reasoning as in establishing Proposition \ref{prop_solutions} and Remark \ref{rem_solutions}, fixing $\varepsilon$ in Proposition \ref{prop_solvability} and $k>n/2+3$, and letting $R=K+k+n+\varepsilon$, we arrive at the following result regarding the construction of complex geometric optics solutions for the Helmholtz equation with uniform constants based on the Gaussian beam quasimodes of Theorem \ref{thm_Gaussian_quasimodes_uniform}.

\begin{prop}
\label{prop_solutions_uniform}
Let  $K\ge 1$ and $\delta>0$. Then for $\lambda\in [1,\infty)\setminus J$ large enough, $|J|\le \delta$, there is $u=u(\cdot; \lambda)\in C^3(M)$ solving $(-\Delta_g-\lambda^2)u=0$ in $M^{\emph{\text{int}}}$ and having the form
\[
u(\cdot;\lambda)=v(\cdot;\lambda)+r(\cdot;\lambda),
\]
where $v(\cdot; \lambda)\in C^\infty(M)$ is the Gaussian beam quasimode constructed in Theorem \ref{thm_Gaussian_quasimodes_uniform} with the constants $C=C(M,g,T,K, \delta_1)$, $\tilde C=\tilde C(M,g,T,K)$, $N=N(M,g,T,K)$, and  $r(\cdot;\lambda)\in C^3(M)$ such that $\|r\|_{C^3(M)}\le C_1 \lambda^{-K}$, as $\lambda\to \infty$ with $C_1=C_1(M,g, \delta, T,  K, \delta_1)>0$.
\end{prop}

\subsection{Main part of the proof of Theorem \ref{thm_main_2}}

Let $T > 0$, $0<\theta_0 <\pi/2$, $0 < r < \frac{\mathrm{Inj}(M)}{2}$ and $c_0 > 0$ be constants given in the condition (H1). Let  $E_{T,\theta_0, r, c_0}$ be the set of points $x\in M^{\text{int}}$ satisfying the condition (H1).  

Based on the proof of Theorem \ref{thm_main}, we begin with the integral identity \eqref{eq_3_4}. We shall show that $p=0$. If $\sup_{x\in M}|p(x)|$ is attained at some point on the boundary of $M$ then by  Proposition \ref{prop_app_boundary_deter} we conclude from  \eqref{eq_3_4} that $p=0$ at this point, and we are done. Thus, we only need to consider the case when  $\sup_{x\in M}|p(x)|$ is not attained at any boundary point, and therefore,  $\sup_{x\in M}|p(x)|$ is achieved at some point in $M^{\text{int}}$. In this case, $\sup_{x\in M}|p(x)|=\sup_{x\in M^{\text{int}}}|p(x)|=\|p\|_{L^\infty(M)}$.
Let $\varepsilon>0$. The fact that $E_{T,\theta_0, r,c_0}$ is dense in $M^{\text{int}}$ and continuity of $p$ implies that there is $ x_0\in E_{T,\theta_0, r,c_0}$ such that
\begin{equation}
\label{eq_7_-2}
|p(x_0)|>\|p\|_{L^\infty(M)}- \varepsilon. 
\end{equation}
As $ x_0\in E_{T,\theta_0, r,c_0}$, there are two non-tangential unit-speed geodesics $\gamma=\gamma_{x_1,w_1}:[0,\tau(x_1,w_1)]\to M$ and $\eta=\eta_{x_2,w_2}:[0,\tau(x_2,w_2)]\to M$, $(x_1,w_1), (x_2,w_2)\in \mathcal{G}_T$, such that $\gamma(t_0)=\eta(\tau_0)=x_0$, $t_0\in (0,\tau(x_1,w_1))$, $\tau_0\in (0,\tau(x_2,w_2))$. Furthermore, both $\gamma$ and $\eta$ adhere to all other requirements outlined in (H1).

Let  $\delta>0$ and $K\ge 1$.  By Proposition \ref{prop_solutions_uniform}, for $\lambda\in [1,\infty)\setminus J$ large enough, $|J|\le \delta$, we have $v^{(1)}, v^{(2)}\in C^3(M)$ solving $(-\Delta_g-\lambda^2)v^{(l)}=0$ in $M^{\emph{\text{int}}}$, $l=1,2$,  of  the form
\begin{equation}
\label{eq_7_0}
v^{(1)}=v+r_1,\quad v^{(2)}=w+r_2,
\end{equation}
where $v=v(\cdot;\lambda), w=w(\cdot; \lambda)\in C^\infty(M)$ are the Gaussian beam quasimodes concentrating near the geodesics $\gamma$ and $\eta$, respectively, given in Theorem \ref{thm_Gaussian_quasimodes_uniform}, and $r_l$, $l=1,2$, satisfies 
\begin{equation}
\label{eq_7_0_2}
\|r_l\|_{L^\infty(M)}\le C_1 \lambda^{-K},
\end{equation}
as $\lambda\to \infty$, with $C_1=C_1(M,g, \delta, T,  K, \delta_1)>0$. It follows from Theorem \ref{thm_Gaussian_quasimodes_uniform} that
\begin{equation}
\label{eq_7_2}
\|v\|_{L^4(M)}\le C, \quad \|w\|_{L^4(M)}\le C, \quad \|v\|_{L^\infty(M)}\le C\lambda^{\frac{n-1}{8}}, \quad \|w\|_{L^\infty(M)}\le C\lambda^{\frac{n-1}{8}},
\end{equation}
as $\lambda\to \infty$, where $C=C(M,g,T,K, \delta_1)>0$. 
We let 
\begin{equation}
\label{eq_7_0_1}
v^{(3)}=\overline{v^{(1)}}, \quad v^{(4)}=\overline{v^{(2)}}.
\end{equation}
Utilizing \eqref{eq_7_0_1},  \eqref{eq_7_0_2}, and \eqref{eq_7_2}, we get 
\begin{equation}
\label{eq_3_12_uniform}
v^{(1)}v^{(2)}v^{(3)}v^{(4)}= |v|^2|w|^2+R,
\end{equation}
where $\|R\|_{L^\infty(M)}\le C\lambda^{-K+\frac{3(n-1)}{8}}$ and the dependence of the implicit constant is as follows: $C=C(M, g, \delta, T , K, \delta_1)$, cf.  \eqref{eq_3_12} and \eqref{eq_3_13}.

Motivated by \eqref{eq_3_17_value_L}, we fix $K=\frac{3(n-1)}{8}+\frac{3}{2}$ so that 
$\|R\|_{L^\infty(M)}\le C\lambda^{-\frac{3}{2}}$ with $C=C(M,g, \delta,T, \delta_1)$. Therefore, by substituting \eqref{eq_3_12_uniform} into the integral identity \eqref{eq_3_4}, we arrive at  
\begin{equation}
\label{eq_3_14_uniform}
\bigg| \int_M p|v|^2|w|^2 dV_g \bigg|\le C\lambda^{-\frac{3}{2}}\|p\|_{L^\infty(M)},
\end{equation}
wherein the implicit constant $C = C(M,g, \delta,T,  \delta_1) > 0$, cf. \eqref{eq_3_14} and \eqref{eq_3_17_value_L}.

Let us mention that, since $K$ is fixed, the only dependence in the constants $C$, $\tilde C$,  and $N$ in Theorem \ref{thm_Gaussian_quasimodes_uniform} is as follows: $C=C(M,g,T,  \delta_1)$, $\tilde C=\tilde C(M,g, T)$,  and $N=N(M,g,T)$. Furthermore, the only dependence in $\delta_1$ in the cutoff function $\rho$ used in the amplitude of the Gaussian beam quasimodes $v$ and $w$ is as follows: $\delta_1=\delta_1(M,g,T)$, see Remark \ref{rem_choice_rho}. 

We have  that the supports of $v$ and $w$ are contained within tubular neighborhoods with radius $\delta_1$ of $\gamma$ and $\eta$, that is, 
\begin{equation}
\label{eq_4_1_delta_1_supp-sep} 
\begin{aligned}
&\mathrm{supp}(v(\cdot; \lambda))\subset \{y\in M: d(y,\gamma([0,\tau(x_1,w_1)]))< \delta_1\},\\
&\mathrm{supp}(w(\cdot; \lambda))\subset \{y\in M: d(y,\eta([0,\tau(x_2,w_2)]))<  \delta_1\},
\end{aligned}
\end{equation}
for all $\lambda\ge 1$.

Next, we will demonstrate that if $\delta_1 < \min\{r/2, \mathrm{Inj}(M)/(1+2c_0)\}$, the condition (iv) in assumption (H1) implies that for all $\lambda \geq 1$,
\begin{equation}
\label{eq_4_1_supp_vw}
\text{supp}(v(\cdot; \lambda) w(\cdot; \lambda)) \subset B_{(1 + 2c_0)\delta_1}(x_0) \cap M. 
\end{equation}
Here, $B_{(1+2c_0)\delta_1}(x_0)$ is an open geodesic ball in $S$ centered at $x_0$ with a radius of $(1 + 2c_0)\delta_1$, where $(S,g)$ is the closed extension of $(M,g)$.

Indeed, if $y \in \mathrm{supp}(v(\cdot; \lambda)w(\cdot; \lambda))\subset \mathrm{supp}(v(\cdot; \lambda))\cap\mathrm{supp}(w(\cdot; \lambda))$ then it follows from \eqref{eq_4_1_delta_1_supp-sep} that $d(y, \gamma(t')) < \delta_1$ and $d(y, \eta(\tau')) < \delta_1$ for some $t' \in [0, \tau(x_1, w_1)]$ and $\tau' \in [0, \tau(x_2, w_2)]$. By the triangle inequality, we obtain $d(\gamma(t'), \eta(\tau')) < 2\delta_1 < r$, and therefore, condition (iv) in assumption (H1) implies that
\begin{equation}
	|t' - t_0| \leq 2c_0 \delta_1.
\end{equation}
Hence, since $\gamma$ is a unit-speed geodesic, we have
\begin{equation}
	d(y, x_0) \leq d(y, \gamma(t')) + d(\gamma(t'), x_0) < \delta_1 + 2c_0 \delta_1.
\end{equation}
Therefore, $y \in B_{(1 + 2c_0)\delta_1}(x_0)$ and this establishes the claim \eqref{eq_4_1_supp_vw}.

If we take $\tau = \tau_0$ in condition (iv) of assumption (H1), we see that if $d(\gamma(t), x_0)  < r$, then $|t - t_0| \leq c_0 d(\gamma(t), x_0)$. Let 
\begin{equation}
\label{eq_4_1_supp_vw_500_1}
\rho = \min\bigg(r, \frac{\mathrm{Inj}(M)}{c_0}\bigg).
\end{equation}
If $d(\gamma(t), x_0) < \rho$  then  $|t - t_0| < \mathrm{Inj}(M)$, and therefore, $d(\gamma(t), \gamma(t_0)) = |t - t_0|$. A similar statement also holds for $\eta$.

We now claim that if $\delta_1 < \frac{\rho}{2}$, the Gaussian beams $v$ and $w$ take the simple form given by \eqref{eq_7_0_1_G_beams} within $B_{\frac{\rho}{2}}(x_0)\cap M$, where $B_{\frac{\rho}{2}}(x_0)$ is an open geodesic ball in $S$ centered at $x_0$ with a radius of $\frac{\rho}{2}$. Indeed, first,  let us show that the set
 \begin{equation}
 \label{eq_4_1_supp_vw_500_2}
 	\{ y \in M : d(y, \gamma(t)) < \delta_1 \text{ for some } t \in [0, \tau(x_1, w_1)], |t - t_0| \geq \rho\} \cap B_{\frac{\rho}{2}}(x_0)
 \end{equation}
 is empty. To prove this, we assume the contrary, that there exists a point $y$ in the intersection of the sets in \eqref{eq_4_1_supp_vw_500_2}. Then, we can find a value $t'$ in the interval $[0, \tau(x_1, w_1)]$ with $|t' - t_0| \geq \rho$ such that $d(y, \gamma(t')) < \delta_1$ and $d(y, x_0) < \frac{\rho}{2}$. Using the triangle inequality and the fact that $\delta_1 < \frac{\rho}{2}$, we obtain $d(\gamma(t'), x_0) < \rho$. Given the equation \eqref{eq_4_1_supp_vw_500_1}, as reasoned in the preceding paragraph, we can conclude that $d(\gamma(t'), x_0) = |t' - t_0| \geq \rho$. This contradiction proves that the set in \eqref{eq_4_1_supp_vw_500_2} is empty.

Moreover, using the same logic, we can show that 
\begin{equation}
\label{eq_4_1_supp_vw_500_3}
\gamma([0,\tau(x_1,w_1)])\cap B_\rho(x_0)\subset \{\gamma(t): t\in [0,\tau(x_1,w_1)], |t-t_0|<\rho \}.
\end{equation}
Since $\rho \leq r < \frac{\mathrm{Inj}(M)}{2}$, the geodesic $\gamma(t)$ does not exhibit self-intersections when $|t - t_0| < \rho$. This, in combination with \eqref{eq_4_1_supp_vw_500_2} and \eqref{eq_4_1_supp_vw_500_3}, establishes the claim for the Gaussian beam quasimode $v$. Similarly, we can demonstrate the same claim for the Gaussian beam quasimode $w$.

Hence, we can redefine $\delta_1$ as $\delta_1 = \delta_1(M, g, T, r, c_0)$ to ensure that it is sufficiently small, such that 
\begin{equation}
\label{eq_4_1_supp_vw_500_4}
\delta_1 < \min\left(\frac{\rho}{2}, \frac{\rho}{2(1+2c_0)}\right). 
\end{equation}
This condition guarantees that 
\begin{equation}
\label{eq_4_1_supp_vw_500_5}
\text{supp}(vw) \subset B_{(1 + 2c_0)\delta_1}(x_0) \cap M \subset B_{\frac{\rho}{2}}(x_0) \cap M.
\end{equation}

According to Theorem \ref{thm_Gaussian_quasimodes_uniform}, see also the proof of \cite[Theorem 6.2]{Ma_Sahoo_Salo_2022}, within $B_{\frac{\rho}{2}}(x_0)\cap M$, the Gaussian beam quasimodes $v$ and $w$ take on a simple form, i.e.
\begin{equation}
\label{eq_7_0_1_G_beams}
v(x;\lambda)=\lambda^{\frac{n-1}{8}} e^{i\lambda \varphi(x)}a(x;\lambda), \quad w(x;\lambda)=\lambda^{\frac{n-1}{8}} e^{i\lambda \psi(x)}b(x;\lambda), 
\end{equation}
where  
\begin{equation}
\label{eq_3_9_uniform}
\begin{aligned}
&\varphi(\gamma(t))=t, \ \nabla\varphi(\gamma(t))=\dot{\gamma}(t), \ \nabla^2 \varphi(\gamma(t))=H_{x_1,w_1}(t)^\flat, \ \|\varphi\|_{C^{3}(\overline{B_{\frac{\rho}{2}}(x_0)\cap M})}\le  \tilde C,\\
&\psi(\eta(\tau))=\tau, \ \nabla\psi(\eta(\tau))=\dot{\eta}(\tau), \ \nabla^2 \psi(\eta(\tau))=H_{x_2,w_2}(\tau)^\flat, \ \|\psi\|_{C^{3}(\overline{B_{\frac{\rho}{2}}(x_0)\cap M})}\le \tilde C.
\end{aligned}
\end{equation}
Here $H_{x_1,w_1}(t)$ and $H_{x_2,w_2}(\tau)$ are symmetric complex $(1,1)$-tensors on $T_{\gamma(t)}M$ and $T_{\eta(\tau)}M$, depending smoothly on $t\in [0,\tau(x_1,w_1)]$ and $\tau\in [0, \tau(x_2,w_2)]$ and satisfying 
\begin{equation}
\label{eq_3_9_uniform_2}
\begin{aligned}
\text{Im}(H_{x_1,w_1}(t)^\flat)\ge 0, \quad \text{Im}(H_{x_1,w_1}(t)^\flat)|_{\dot{\gamma}(t)^\perp}\ge {\tilde C}^{-1}g,\\
\text{Im}(H_{x_2,w_2}(\tau)^\flat)\ge 0, \quad \text{Im}(H_{x_2,w_2}(\tau)^\flat)|_{\dot{\eta}(\tau)^\perp}\ge {\tilde C}^{-1} g,
\end{aligned}
\end{equation}
respectively. 
We also have in $B_{\frac{\rho}{2}}(x_0)\cap M$, 
\begin{equation}
\label{eq_3_9_uniform_3}
a(\cdot; \lambda)=\sum_{j=0}^N \lambda^{-j}\tilde a_j,\quad \tilde a_j=a_j\rho, \quad b(\cdot; \lambda)=\sum_{j=0}^N \lambda^{-j}\tilde b_j,\quad \tilde b_j=a_j\rho, 
\end{equation}
and 
\begin{equation}
\label{eq_7_1}
\|a_j\|_{C^{3}(\overline{B_{\frac{\rho}{2}}(x_0)\cap M})}\le  \tilde C, \quad \|b_j\|_{C^{3}(\overline{B_{\frac{\rho}{2}}(x_0)\cap M})}\le  \tilde C,  \quad j=0,\dots, N.
\end{equation}
Furthermore, 
\begin{equation}
\label{eq_4_0_-1_a00_new}
\begin{aligned}
a_{0}(\gamma(t))=\exp\bigg[-\frac{1}{2}\int_0^t \emph{\text{tr}}_g(H_{x_1,w_1}(s))ds\bigg], \\ 
b_{0}(\eta(\tau))=\exp\bigg[-\frac{1}{2}\int_0^{\tau} \emph{\text{tr}}_g(H_{x_2,w_2}(s))ds\bigg].
\end{aligned}
\end{equation}
The implicit constants in  \eqref{eq_3_9_uniform}, \eqref{eq_3_9_uniform_2},  \eqref{eq_3_9_uniform_3}, and \eqref{eq_7_1} are expressed as $\tilde C=\tilde C(M,g,T)>0$ and $N=N(M,g,T)\in \N$.

Now in  $B_{\frac{\rho}{2}}(x_0)\cap M$, in view of \eqref{eq_7_0_1_G_beams},  we have
\[
|v|^2|w|^2=\lambda^{\frac{n-1}{2}}e^{-2\lambda(\text{Im}\, \varphi+\text{Im}\, \psi)}|a|^2|b|^2. 
\]
Letting 
\[
\Psi=2(\text{Im}\, \varphi+\text{Im}\, \psi)\ge 0,
\]
we deduce from equations \eqref{eq_3_9_uniform} and \eqref{eq_3_9_uniform_2} the following properties:
\begin{equation}
\label{eq_7_3}
\Psi(x_0)=0, \quad d\Psi(x_0)=0, \quad \nabla^2\Psi(x_0)\ge c g, 
\end{equation}
where $c = c(M, g, T,  \theta_0) > 0$. To show the last inequality in \eqref{eq_7_3}, first in Fermi coordinates along $\gamma$ near $x_0$, we have
\begin{equation}
	\text{Im}(\nabla^2 \varphi(\gamma(t)))^{\sharp} =
	\begin{pmatrix}
		0 & 0 \\
		0 & \text{Im}(H_{x_1, w_1}(t))\vert_{\dot{\gamma}(t)^{\perp}}
	\end{pmatrix}
\end{equation}
where the block decomposition of the matrix is understood with respect to the decomposition $T_{\gamma(t)} M = \R\dot{\gamma}(t) \oplus (\dot{\gamma}(t))^\perp$.  

We have a similar decomposition for $\text{Im}(\nabla^2 \psi(\eta(\tau)))^\sharp$ in Fermi coordinates along $\eta$. Let $t_0$ and $\tau_0$ be the unique times such that $\gamma(t_0) = \eta(\tau_0) = x_0$, and let $W = (\dot{\gamma}(t_0))^{\perp} \cap (\dot{\eta}(\tau_0))^\perp \subset T_{x_0} M^{\text{int}}$. We can write any $v \in T_{x_0} M^{\text{int}}$ uniquely as $v = v_\gamma \dot{\gamma}(t_0) + v_\eta \dot{\eta}(\tau_0) + w$ with $v_\gamma, v_\eta \in \R$, and $w \in W$. By using that $\dot{\gamma}(t_0)$ is in the kernel of $\text{Im}(\nabla^2 \varphi(x_0))^\sharp$ and $\dot{\eta}(\tau_0)$ is in the kernel of $\text{Im}(\nabla^2 \psi(x_0))^\sharp$, we can expand
\begin{equation}\label{eq_nabla2Psi}
\begin{aligned}
	(\nabla^2 \Psi(x_0)^\sharp v, v) &= (2\text{Im}(\nabla^2 \varphi(x_0))^\sharp(v_\eta \dot{\eta}(\tau_0) + w), v_\eta \dot{\eta}(\tau_0) + w) \,+ \\
	& \quad + (2\text{Im}(\nabla^2 \psi(x_0))^\sharp(v_\gamma \dot{\gamma}(t_0) + w), v_\gamma \dot{\gamma}(t_0) + w).
\end{aligned}
\end{equation}
Here and in what follows $(\cdot,\cdot)=(\cdot, \cdot)_g$. 
We can decompose $v_\eta \dot{\eta}(\tau_0) + w$ with respect to $\R\dot{\gamma}(t_0) \oplus (\dot{\gamma}(t_0))^{\perp}$ as
\begin{equation}
\label{eq_nabla2Psi_10}
	v_\eta \dot{\eta}(\tau_0) + w = (v_\eta \cos\theta) \dot{\gamma}(t_0) + \tilde{w}
\end{equation}
for some unique $\tilde{w} \in (\dot{\gamma}(t_0))^{\perp}$ with $\cos \theta = (\dot{\gamma}(t_0), \dot{\eta}(\tau_0))$. Taking norms on both sides of \eqref{eq_nabla2Psi_10}, we see that $|\tilde{w}|^2 = v_\eta^2 \sin^2 \theta + |w|^2$ since $w$ is orthogonal to $\dot{\eta}(\tau_0)$. Moreover, as $\text{Im}(\nabla^2 \varphi(x_0))^\sharp \dot{\gamma}(t_0) = 0$, we get
\[
	(\text{Im}(\nabla^2 \varphi(x_0))^\sharp(v_\eta \dot{\eta}(\tau_0) + w), v_\eta \dot{\eta}(\tau_0) + w) = (\text{Im}(\nabla^2 \varphi(x_0))^\sharp \tilde{w}, \tilde{w}) \geq {\tilde C}^{-1} |\tilde{w}|^2
\]
by \eqref{eq_3_9_uniform_2}, since $\tilde{w} \in (\dot{\gamma}(t_0))^\perp$.  By working analogously with the second term in \eqref{eq_nabla2Psi} and combining, we get
\[
	(\nabla^2 \Psi(x_0)^\sharp v, v) \geq 2 {\tilde C}^{-1}((v_\gamma^2 + v_\eta^2)\sin^2\theta + 2 |w|^2).
\]
As $|v|^2 = v_\gamma^2 + v_\eta^2 +2v_\gamma v_\eta \cos\theta + |w|^2$ and $v_\gamma^2 + v_\eta^2 \geq \frac{1}{2}(v_\gamma^2 + v_\eta^2 + 2v_\gamma v_\eta \cos\theta)$, it follows that $\nabla^2 \Psi (x_0) \geq ({\tilde C}^{-1} \sin^2 \theta)g \geq ({\tilde C}^{-1} \sin^2 \theta_0)g$, since $\theta \in [\theta_0, \pi/2]$.  This shows the last inequality in \eqref{eq_7_3}. 

 Now if $z=(z_1,\dots, z_n)$ are normal coordinates centered at $x_0\in M^{\text{int}}$, we have 
\[
\Psi(z)=\frac{1}{2}\Psi''(0)z\cdot z+\Psi_3(z), 
\]
where $\Psi''(0)\ge cI$, $c = c(M, g, T,  \theta_0) > 0$, in view of the last inequality in \eqref{eq_7_3}. Thanks to \eqref{eq_3_9_uniform}, we get $|\Psi_3(z)|\le \tilde C |z|^3$ for $z\in \exp_{x_0}^{-1}( B_{\frac{\rho}{2}}(x_0)\cap M)\subset B_{\frac{\rho}{2}}(0)$ with $\tilde C=\tilde C(M,g, T)>~0$. Here the exponential map $\exp_{x_0}: B_{\frac{\rho}{2}}(0) \to B_{\frac{\rho}{2}}(x_0)$ is a diffeomorphism.

Now, if $0 < \tilde \delta < \frac{\rho}{2}$ is small enough to satisfy the condition
\begin{equation}
\label{eq_7_4}
\tilde C\tilde \delta\le c/4,
\end{equation}
then we have
\begin{equation}
\label{eq_7_5}
\Psi(z)\ge \frac{1}{2}c|z|^2-\tilde C |z|^3\ge (c/2-\tilde C \tilde \delta)|z|^2\ge \frac{c}{4}|z|^2, \  z\in \exp_{x_0}^{-1}( B_{\tilde \delta}(x_0)\cap M)\subset B_{\tilde \delta}(0).
\end{equation}
Note that  $\tilde \delta=\tilde \delta(\tilde C,c,\rho)=\tilde \delta(M,g,T, \theta_0, r,c_0)$.

In light of \eqref{eq_7_4} and \eqref{eq_4_1_supp_vw_500_4}, we can redefine $\delta_1 = \delta_1(M, g, T, \theta_0, r, c_0) > 0$ to be sufficiently small, ensuring that
\begin{equation}
\label{eq_7_5_new_500_1}
\delta_1\le \min\left(\frac{\rho}{2}, \frac{\rho}{2(1+2c_0)}, \frac{c}{4\tilde C (1+2c_0)}\right).
\end{equation}
By letting $\tilde \delta := (1 + 2c_0)\delta_1$, we can see from \eqref{eq_7_5_new_500_1} and \eqref{eq_4_1_supp_vw_500_5} that
\[
\text{supp}(vw) \subset B_{\tilde \delta}(x_0) \cap M \subset B_{\frac{\rho}{2}}(x_0) \cap M,
\]
and \eqref{eq_7_5} holds. It is essential to emphasize that the independence of the constant $\tilde C$ from $\delta_1$ in Theorem \ref{thm_Gaussian_quasimodes_uniform}  is crucial for \eqref{eq_7_5_new_500_1}. 

This choice of $\delta_1$ will remain fixed in the subsequent discussion. As a consequence, the constants, previously dependent on $\delta_1$, will now depend on $M$, $g$, $T$, $\theta_0$, $r$, and $c_0$. 

We shall proceed to thoroughly examine the rest of the proof of Theorem \ref{thm_main}, with a specific focus on controlling the dependencies of all the constants involved. Throughout this examination, we will represent all constants as $C$, and it is important to note that these constants might change from line to line.

Using the fact that the support of the product $|v|^2|w|^2$ is entirely confined within $B_{\tilde \delta}(x_0) \cap M$  and \eqref{eq_7_1}, we obtain \eqref{eq_3_15}, where the implicit constant is denoted as $C = C(M,g, T) > 0$.

Now, \eqref{eq_3_14_uniform} and \eqref{eq_3_15} imply that 
the bound \eqref{eq_3_17} holds with $L=3/2$, and the implicit constant is denoted as $C = C(M,g,\delta, T, \theta_0, r, c_0) > 0$. Moreover, the integrals in \eqref{eq_3_17} are taken over $B_{\tilde \delta}(x_0)\cap M$. By using \eqref{eq_7_5}, we obtain \eqref{eq_3_17_new}, where the implicit constant is expressed as $C = C(M,g, T, \theta_0, r, c_0) > 0$. Thus, we get \eqref{eq_3_19} with an implicit constant of $C = C(M,g,\delta, T, \theta_0,r, c_0) > 0$, and with the integral taken over $B_{\tilde \delta}(x_0)\cap M$.

Assume first that $B_{\tilde \delta}(x_0)\subset M^{\text{int}}$. Thanks to \eqref{eq_7_4}, we can now apply the rough stationary phase Lemma \ref{lem_stationary_phase} to the integral in \eqref{eq_3_19}. This allows us to obtain \eqref{eq_3_19_new}, where the implicit constant is $C = C(M,g, T, \theta_0,r,c_0) > 0$. If $B_{\tilde \delta}(x_0)\cap \partial M \neq \emptyset$, then first, by Proposition \ref{prop_app_boundary_deter}, we conclude from \eqref{eq_3_4} that $p|_{\partial M}=0$. Extending $\tilde p:=p|\tilde a_0|^2|\tilde b_0|^2$ by zero to $B_{\tilde \delta}(x_0)\setminus (B_{\tilde \delta}(x_0)\cap M)$ and denoting this extension by $\tilde p$ again, we see that $\tilde p\in C^{0,\alpha}(\overline{B_{\tilde \delta}(x_0)})$, and $\text{supp}(\tilde p)\subset B_{\tilde\delta}(x_0)$ is compact, see Lemma \ref{lem_extension_by_zero}. In view of Remark \ref{rem_stationary_phase}, applying Lemma \ref{lem_stationary_phase} to the integral in \eqref{eq_3_19}, we obtain \eqref{eq_3_19_new}, where the implicit constant is $C = C(M,g, T, \theta_0,r,c_0) > 0$. Consequently, in both cases, we arrive at \eqref{eq_3_20_new}, with the implicit constant expressed as $C = C(M,g,\delta, T, \theta_0,r, c_0) > 0$.

Next, it follows from  \eqref{eq_4_0_-1_a00_new} that 
\[
a_{0}(x_0)=\exp\bigg[-\frac{1}{2}\int_0^{t_0} \emph{\text{tr}}_g(H_{x_1,w_1}(s))ds\bigg],
\]
where $0<t_0<\tau(x_1,w_1)\le T$. Working with the open cover of $\gamma$ consisting of $\tilde N=\tilde N(M,g,T)$ elements, and using \eqref{eq_4_0_-1_phase_l} in each element of the cover,  we conclude  that $|a_0(x_0)| \ge C$, where $C = C(M,g, T) > 0$. We refer to the proof of \cite[Theorem 6.2]{Ma_Sahoo_Salo_2022} for the existence of such a cover. Similarly, we have $|b_0(x_0)| \ge C$.  Furthermore,  $|\det \nabla^2\Psi(x_0)| \le C$, where $C = C(M,g, T) > 0$, due to \eqref{eq_3_9_uniform}. Hence,  from \eqref{eq_3_20_new} and the fact that $p\in \mathcal{A}(B)$, we deduce that
\begin{equation}
\label{eq_7_6}
|p(x_0)|\le CB\lambda^{-\alpha/2}\|p\|_{L^\infty(M)},
\end{equation}
for $\lambda\in [1,\infty)\setminus J$ large enough, satisfying the assumption (A). Here, $C=C(M,g,\delta, T, \theta_0,r, c_0)>0$.

By combining \eqref{eq_7_6} with \eqref{eq_7_-2}, we obtain the inequality
\[
\|p\|_{L^\infty(M)}(1- CB\lambda^{-\alpha/2}) <\varepsilon, 
\]
for all $\varepsilon>0$. Now, by letting $\varepsilon\to 0$ and choosing $\lambda=\lambda(M,g,\delta, T, \theta_0, B, r, c_0)\in [1,\infty)\setminus J$ large enough, satisfying the assumption (A), we conclude that $p=0$. This completes the proof of Theorem \ref{thm_main_2}.

\section{Proof of Theorem \ref{thm:intersection}}

\label{sec:intersection}

\subsection{Existence of non-tangential geodesics between boundary points}

We start by showing that almost all points in $M^{\mathrm{int}}$ lie on a non-tangential geodesic. We will follow the approach used in \cite[Lemma 3.1]{salo2017calderon}, but modify it as we do not want to require strict convexity of the boundary.

\begin{lem} \label{lem:geodesic}
Let $(M, g)$ be a smooth compact Riemannian manifold of dimension $n\ge 2$ with smooth boundary. Then almost every point of $M^{\mathrm{int}}$ lies on some non-tangential geodesic.	
\end{lem}

Before giving the proof of Lemma \ref{lem:geodesic}, let us do some preparations.
Let $(M, g)$ be embedded in a closed manifold $(N, g)$. Let $SM^{\mathrm{int}}$ and $SN$ denote the unit sphere bundle of $M^{\textrm{int}}$ and $N$ respectively and let $\varphi_t$ denote the geodesic flow on $SN$. For $(x, v) \in SM^{\mathrm{int}}$, we define the future and past exit times respectively as
\begin{align*}
	l_{+}(x,v) &= \sup \, \{T > 0: \varphi_t(x,v) \in SM^{\mathrm{int}} \text{ for all } 0 \leq t < T\}, \\
	l_{-}(x,v) &= \inf \, \{T < 0: \varphi_t(x,v) \in SM^{\mathrm{int}} \text{ for all } T < t \leq 0\}.
\end{align*}
Consider the following sets of good and bad directions,
\begin{align*}
	G &= \{(x,v) \in SM^{\mathrm{int}} : l_-(x,v) \text{ and } l_+(x,v) \text{ are finite with transversal}\\
	& \text{       intersections}\},\\
	B_1 &= \{(x,v) \in SM^{\mathrm{int}} : \text{exactly one of } l_\pm(x,v) \text{ is finite}\},\\
	B_2 &= \{(x,v) \in SM^{\mathrm{int}} : l_-(x,v) \text{ and } l_+(x,v) \text{ are infinite}\},\\
	B_3 &= \{(x,v) \in SM^{\mathrm{int}}: l_-(x,v) \text{ and } l_+(x,v) \text{ are finite with a tangential}\\ & \text{          intersection}\}.
\end{align*}
In the definition of $G$, by transversal intersections, we mean that the unit speed geodesic going through $(x,v)$ at time $0$ intersects $\p M$ transversally at times $l_-(x,v)$ and $l_+(x,v)$. On the opposite, in the definition of $B_3$, by a tangential intersection, we mean that the unit speed geodesic going through $(x,v)$ at time $0$ intersects $\p M$ tangentially at least once at times $l_-(x,v)$ or $l_+(x,v)$.

Let $m$ denote the Riemannian volume measure on $(M, g)$ and let $\mu$ denote the Liouville measure on $SN$, see \cite[Section 3.6.2]{PSU_book}.
\begin{lem}
\label{lem_measure_B}
	The sets $B_1$ and $B_3$ have Liouville measure $0$.	
\end{lem}

\begin{proof}
It follows directly from \cite[Lemma 3.2]{salo2017calderon} that $\mu(B_1) = 0$. For $B_3$, let 
\[
	E = \{\varphi_t(x,v) : (x,v) \in \partial_0 SM \text{ and } 0 \leq t < 1\},
\]
where $\partial_0 SM:=S\p M$. We have $\mu(E) = 0$ as $\mathrm{dim}(\p_0 SM) = 2n - 3$ and $\mathrm{dim}(SN) = 2n - 1$. By invariance of the Liouville measure, $\mu(\varphi_t(E)) = 0$ for all $t \in \R$. If $(x, v) \in B_3$, the unit speed geodesic going through $(x,v)$ at time $0$ intersects $\p M$ tangentially at least once at times $l_-(x,v)$ or $l_+(x,v)$, that is, $\varphi_{l_-(x,v)}(x,v) \in \p_0 SM$ or $\varphi_{l_+(x,v)}(x,v) \in \p_0 SM$. It follows that $B_3 \subset \bigcup_{k = -\infty}^\infty \varphi_k(E)$ and so $\mu(B_3) = 0$.
\end{proof}

\begin{proof}[Proof of Lemma \ref{lem:geodesic}]
Consider the set
\begin{align*}
A =& \{x \in M^\mathrm{int} : \text{for any } v \in S_x M^\mathrm{int},\text{ either } l_-(x,v) = -\infty, \text{ or } l_+(x,v) = \infty, \\
&\text{ or }l_-(x,v) \text{ and } l_+(x,v) \text{ are finite with a tangential} \text{ intersection}\}.
\end{align*}
Here, as before, by a tangential intersection, we mean that the unit speed geodesic going through $(x,v)$ at time $0$ intersects $\p M$ tangentially at least once at times $l_-(x,v)$ or $l_+(x,v)$.
The claim will be proved if we can show $m(A) = 0$.

Suppose that $A$ has positive measure. By the Lebesgue density theorem, we can find a point $x_0 \in A$ such that
\[
\lim_{\varepsilon \to 0} \frac{m(A \cap B_\varepsilon(x_0))}{m(B_\varepsilon(x_0))} = 1.
\]
By \cite[Lemma 2.10]{KKL01}, there exists
$v_0 \in S_{x_0} M^{\mathrm{int}}$ such that the geodesic $\gamma_{x_0, v_0}$ minimises the distance between $x_0$ and $\p M$. As $\p M$ is closed, such a $v_0$ always exists, although it might not be unique. By minimality, as $\p M$ is smooth, this geodesic exits $M$ normally. Therefore, there is $v_0 \in S_{x_0} M^{\mathrm{int}}$ such that $l_+(x_0, v_0) < \infty$ and the geodesic through $(x_0, v_0)$ exits $M$ normally.

The implicit function theorem guarantees the existence of a neighborhood $U$ of $(x_0, v_0)$ in $SM^{\mathrm{int}}$ for which $l_+(x, v) < \infty$ and the geodesic through $(x, v)$ exits $M$ transversally for all $(x,v) \in U$. Following the notation in \cite[Section 3]{salo2017calderon},  we can find $\varepsilon_0$ and sets $S_{B_{\varepsilon}(x_0), W} \subset U$, $\varepsilon<\varepsilon_0$, for which
\[
(x,v) \in S_{B_{\varepsilon}(x_0), W} \cap SA \quad\Longrightarrow \quad (x,v) \in (B_1 \cup B_3) \cap SA.
\]
Since by Lemma \ref{lem_measure_B} $\mu(B_1 \cup B_3) = 0$, it follows by the same reasoning as in \cite[the proof of Lemma 3.1]{salo2017calderon} that $m(A \cap B_\varepsilon(x_0)) = 0$, $\varepsilon<\varepsilon_0$, contradicting the fact that $x_0$ was a point of density one in $A$.
\end{proof}

\subsection{Perturbation of pairs of geodesics}

For $(x,v) \in SM^{\mathrm{int}}$, we denote by $\gamma_{x,v} : [-T_1, T_2]\to M$, $0<T_1,T_2<\infty$, the unique unit-speed geodesic such that $\gamma_{x,v}(0) = x$ and $\dot{\gamma}_{x,v}(0) = v$. We say that the smooth family of geodesics $(\gamma_s)_{s \in (-\varepsilon, \varepsilon)}$, is a variation through geodesics of $\gamma_{x,v}$ with variation field $J$ if $\gamma_0 = \gamma_{x, v}$ and $J(t) = \frac{d}{ds} \big\vert_{s=0} \gamma_s(t)$.

\begin{lem}\label{lem:variation}
Let $\gamma_{x,u}:[-T_u, T_u] \to M$ and $\gamma_{x,v}: [-T_v, T_v] \to M$ be geodesic segments intersecting only at $x$ and let $\gamma_s$ be a variation through geodesics of $\gamma_{x,v}$ with variation field $J$. Suppose there are nonzero sequences $s_j, t_j \to 0$ such that $\gamma_{s_j}$ intersects $\gamma_{x,u}$ at the point $\gamma_{s_j}(t_j)$ for all $j \in \N$. Then,
\[
	J(0) \in \mathrm{span}(\{u, v\}) = \mathrm{span}(\{\dot{\gamma}_{x,u}(0), \dot{\gamma}_{x,v}(0)\}).
\]
\end{lem}

\begin{proof}
Let $(\tau, y)$ be Fermi coordinates about $x$ adapted to the geodesic $\gamma_{x,u}$. The coordinates induce the frame $\frac{\p}{\p \tau}, \frac{\p}{\p y^1}, \dots, \frac{\p}{\p y^{n-1}}$. In these coordinates, $\gamma_{x,u}$ is given by $\{y = 0\}$ and hence $u = \dot{\gamma}_{x,u}(0) = \frac{\p}{\p \tau}$. We can express $\gamma_s(t)$ as $(\tau_s(t), y_s(t))$. Let $\tilde{J}(t)$ and $\tilde{\dot{\gamma}}_{x,v}(t)$ be the respective components of $J(t)$ and $\dot{\gamma}_{x,v}(t)$ in the $y$ coordinates. We have
\[
	y_s(t) = y_0(t) + s \tilde{J}(t) + O(s^2).
\]
Since $\gamma_{s_j}$ intersects $\gamma_{x,u}$ at the point $\gamma_{s_j}(t_j)$, we have $y_{s_j}(t_j) = 0 = y_0(0)$. Plugging this in the above expression and rearranging gives
\begin{equation}
	\frac{y_0(0) - y_0(t_j)}{t_j} \frac{t_j}{s_j} = \tilde{J}(t_j) + O(s_j).
\end{equation}
The right-hand side converges to $\tilde{J}(0)$ as $j \to \infty$. As for the left-hand side, the first fraction converges to $-\tilde{\dot{\gamma}}_{x,v}(0)$. It follows that $t_j/s_j$ has to converge since $\tilde{\dot{\gamma}}_{x,v}(0)$ is nonzero, and so $\tilde{J}(0)$ is a multiple of $\tilde{\dot{\gamma}}_{x,v}(0)$. As $u = \frac{\p}{\p \tau}$ generates the first components of both $J(0)$ and $v$, the result follows.
\end{proof}

\begin{lem} 
\label{lem:perturb}
	Let $(M, g)$ be a smooth compact Riemannian manifold of dimension $n\ge 3$ with smooth boundary.
	Let $x \in M^{\mathrm{int}}$ be such that there exist distinct directions $u, v \in S_x M^{\mathrm{int}}$ such that the geodesics $\gamma_{x, u}$ and $\gamma_{x, v}$ are non-tangential.
	Suppose that the order of conjugacy of $x$ is at most $n-3$. Then, there is a sequence $(v_n)_{n \in \N} \subset  S_x M^{\mathrm{int}}$, $v_n \to v$, such that for all $n\in \N$ sufficiently large,  $\gamma_{x, v_n}$ does not intersect $\gamma_{x, u}$ (except at $x$) and $\gamma_{x, v_n}$ does not self-intersect at $x$.
\end{lem}

\begin{rem}
	The non-intersection result of Lemma \ref{lem:perturb} is global, while that of Lemma \ref{lem:variation} is only local.
\end{rem}

The main idea of the proof of Lemma \ref{lem:perturb} is to find a direction towards which we can perturb $\gamma_{x,v}$ away from $\gamma_{x,u}$. We wish to use Lemma \ref{lem:variation}. Say that both geodesics intersect at $q \in M$, that is, we have $t$ and $r$ such that $\gamma_{x,v}(t) = \gamma_{x,u}(r) = q$. If we can find a variation through geodesics $\gamma_s$ of $\gamma_{x, v}$ with variation field $J$ such that $J(t)$ is transversal to $\mathrm{span}(\{\dot{\gamma}_{x,v}(t), \dot{\gamma}_{x,u}(r)\})$, then Lemma \ref{lem:variation} guarantees that $\gamma_s$ and $\gamma_{x, u}$ do not intersect for all nonzero $s$ small enough. The condition on the order of conjugacy allows us to do this around every point of intersection between $\gamma_{x, v}$ and $\gamma_{x, u}$. The sequence $v_n$ is then obtained from the different values of $\dot{\gamma}_s(0)$.

\begin{proof}
To simplify notation, we write $\gamma_{x, v} = \gamma_v$ as all the geodesics will start at $x$. Let $\gamma_{v}: [-T_1, T_2] \to M$ be maximal and intersect $\gamma_{u}$ at the points $q_j$ at times $(t_j, r_j)$ with 
\[
-T_1 \leq t_{-k_-} \leq \dots \leq t_{-1} < 0 < t_1 \leq \dots \leq t_{k_+} \leq T_2
\]
and $r_j < r_{j+1}$ whenever $t_j = t_{j+1}$, that is, $\gamma_v(t_j) = \gamma_u(r_j) = q_j$. Without loss of generality, we can assume that the geodesics do not intersect at the boundary. Otherwise, we can make $M$ slightly larger and extend the geodesics by adding a small open neighborhood of $\p M$ in $N$. We denote  $v_j = t_j v \in T_x M^{\mathrm{int}}$  so that $\exp_x (v_j) = \gamma_v(t_j)$. The behaviour of the differential of the exponential map at $v_j \in  T_x M^{\mathrm{int}}$, $d(\exp_x)_{v_j} : T_{v_j} T_x M^{\mathrm{int}} \to T_{q_j} M^\textrm{int}$ is closely related to the order of conjugacy of $M$ at $x$. Indeed, denoting the restriction of $\gamma_{v}$ to $[0, t_j]$ as $\gamma_j$, we have
\[
\dim \ker d(\exp_x)_{v_j} = \mathrm{conj}_{\gamma_j}(x, q_j),
\]
see \cite[Proposition 3.5, page 117]{do_Carmo_1992}.

Consider the linear spaces
\[
E_j = \{\xi \in T_{v_j} T_x  M^\textrm{int} : d(\exp_x)_{v_j}(\xi)\text{ is parallel to }\dot{\gamma}_{u} (r_j)\}.
\]
Note that the kernel of $d(\exp_x)_{v_j}$ is a subspace of $E_j$ and it is the whole space unless $\dot{\gamma_{u}}(r_j)$ is in the range of $d(\exp_x)_{v_j}$, in which case the dimension of $E_j$ is $\mathrm{conj}_{\gamma_j}(x, q_j) + 1$. Through the identification of $T_{v_j} T_x M^\textrm{int}$ with $T_x  M^\textrm{int}$, we can identify $E_j$ with the affine plane 
\[
P_j = \{v_j + \xi : \xi \in E_j\} \subset T_x M^\textrm{int}.
\]
Because the order of conjugacy at $x$ is at most $n-3$, the dimension of the affine plane $P_j$ is at most $n-2$. Moreover, the planes $P_j$ are transversal to $v_j$.

Let $\pi : T_x M^\textrm{int}\to S_x M^\textrm{int}$ be the natural projection from the tangent space to the unit sphere bundle at $x$. For $j \geq 1$, we have $\pi(v_j) = v$ while for $j \leq -1$, $\pi(v_j) = -v$. In any case, near $\pi(v_j)$, the image of $P_j$ through $\pi$ is an immersed submanifold $\tilde{P}_j \subset S_x  M^\textrm{int}$ of codimension at least $1$. Hence, $T_v \tilde{P}_j$ is a proper subspace of $T_v S_x M^\textrm{int}$ and since a vector space over an infinite field cannot be realized as the finite union of proper subspaces, there is $\alpha \in T_v S_x M^\textrm{int}$ such that $\alpha \not\in \cup_{j=1}^{k_+} T_v \tilde{P}_j$. Similarly, we may also require that $-\alpha \not\in \cup_{j=1}^{k_-} T_{-v} \tilde{P}_{-j}$, where $-\alpha \in T_{-v}S_x M^\textrm{int}$ is obtained through the identification with $T_x  M^\textrm{int}$. 

We claim that
\[
v_n = \exp_v^S(\alpha/n)
\]
is the desired sequence, where $\exp_v^S$ is the exponential map on $T_v S_x M^\textrm{int}$ taking values in the sphere $S_x M^\textrm{int}$.

As $n$ goes to $\infty$, it is clear that $v_n$ converges to $v$. It remains to show that $\gamma_{v_n}$ does not intersect $\gamma_u$ for all $n$ sufficiently large. Consider the variation through geodesics of $\gamma_v$,
\[
\gamma_s(t) = \exp_x(t\exp_v^S(s\alpha)).
\]
Note that $\gamma_s$ is non-tangential for sufficiently small values of $s$ as it is a variation of the non-tangential geodesic $\gamma_v$. Since $\exp_v^S(s\alpha)$ is equal to $v + s\alpha$ to first order, we have
\[
J(t) = \frac{d}{ds} \Big\vert_{s=0} \gamma_s(t) = \frac{d}{ds} \Big\vert_{s=0} \exp_x(t(v + s\alpha)) = d (\exp_x)_{tv}(t\alpha)
\]
which is a normal Jacobi field along $\gamma_v$. 

For $\delta > 0$, consider the pairs of geodesic segments $\gamma_v \vert_{[t_j - \delta, t_j + \delta]}$ and $\gamma_u \vert_{[r_j - \delta, r_j + \delta]}$. Choosing $\delta$ small enough, we can guarantee that all such pairs intersect exactly once at $q_j$. By compactness, there is $\delta > 0$ such that for all small enough variations $s$, the only possible intersections between $\gamma_s$ and $\gamma_u$ have to occur between the segments $\gamma_s\vert_{[t_j - \delta, t_j + \delta]}$ and $\gamma_u\vert_{[r_j - \delta, r_j + \delta]}$. We can see $\gamma_s\vert_{[t_j - \delta, t_j + \delta]}$ as a variation through geodesics of $\gamma_v \vert_{[t_j - \delta, t_j + \delta]}$ with variation field $J \vert_{[t_j - \delta, t_j + \delta]}$. At the points $t = t_j$ where $\gamma_v$ intersects $\gamma_u$, we have
\[
	J(t_j) = d(\exp_x)_{v_j}(t_j \alpha).
\]
By the choice of $\alpha$, we know that $t_j \alpha$ is not in $E_j$. Hence, $J(t_j)$ is nonzero and transversal to $\mathrm{span}(\{\dot{\gamma}_v(t_j), \dot{\gamma}_u(r_j)\})$.

We claim that there is some $\varepsilon_j > 0$ such that $\gamma_s \vert_{[t_j - \delta, t_j + \delta]}$ does not intersect $\gamma_u \vert_{[r_j - \delta, r_j + \delta]}$ for all $s \in (-\varepsilon_j, \varepsilon_j) \setminus \{0\}$. If such a $\varepsilon_j$ did not exist, there would be a sequence $s_l \to 0$ and times $t_{j, l} \in [t_j - \delta, t_j + \delta]$ such that $\gamma_{s_l}$ intersects $\gamma_u\vert_{[r_j - \delta, r_j + \delta]}$ at the point $\gamma_{s_l}(t_{j, l})$. By compactness, $t_{j, l}$ would admit a convergent subsequence which must converge to $t_j$ by continuity and because it is the only time when $\gamma_v \vert_{[t_j - \delta, t_j + \delta]}$ intersects $\gamma_u \vert_{[r_j - \delta, r_j + \delta]}$. But then, by Lemma \ref{lem:variation}, this would imply $J(t_j) \in \mathrm{span}(\{\dot{\gamma}_v(t_j), \dot{\gamma}_u(r_j)\})$ which contradicts the above choice of $\alpha$.

As we accounted for all the pairs of segments that could intersect and there are finitely many such pairs, we know that $\gamma_s$ does not intersect $\gamma_u$ (except at $x$ at time $0$) for all nonzero $s$ small enough. Moreover, for such $s$, $\gamma_s$ does not self-intersect at $x$ as it would then intersect $\gamma_u$. The claim follows since taking $s$ small corresponds to taking $n$ large enough.
\end{proof}

\subsection{Proof of Theorem \ref{thm:intersection}}
By Lemma \ref{lem:geodesic}, almost every point of $M^{\mathrm{int}}$ lies on some non-tangential geodesic. Let $x \in M^{\mathrm{int}}$ be such a point and let $\gamma_{x,v}$ be a non-tangential geodesic passing through $x$. As $\gamma_{x,v}$ is non-tangential, we can perturb $v$ in some open neighborhood to get another non-tangential geodesic $\gamma_{x, u}$ passing through $x$. By Lemma \ref{lem:perturb}, we can perturb $\gamma_{x, v}$ to obtain $\gamma_1$ that intersects $\gamma_{x,u}$ only at $x$ and does not self-intersect at $x$. If $\gamma_{x,u}$ does not self-intersect at $x$, then we set $\gamma_2 = \gamma_{x,u}$ and we are done. Otherwise, we can again apply Lemma \ref{lem:perturb} to get a geodesic $\gamma_2$ that only intersects $\gamma_1$ at $x$ and does not self-intersect at $x$. The proof of Theorem \ref{thm:intersection} is complete.

\begin{appendix}

\section{A rough stationary phase argument}

\label{sec_rough_stationary_phase}

Let $U\subset \R^n$ be an open set and let $0<\alpha<1$. We define the space $C^{0,\alpha}(\overline{U})$  of the H\"older continuous functions to be the subspace of $ C(\overline{U})$ consisting of those functions  $u$ for which $\sup_{x, y\in U, x\ne y}\frac{|u(x)-u(y)|}{|x-y|^\alpha}$ is finite. The space $C^{0,\alpha}(\overline{U})$ is a Banach space with norm given by
\[
\|u\|_{C^{0,\alpha}(\overline{U})}=\sup_{x, y\in U, x\ne y}\frac{|u(x)-u(y)|}{|x-y|^\alpha}+\|u\|_{L^\infty(U)}.
\]
We note that $C^{0,\alpha}(\overline{U})$ is an algebra under pointwise multiplication, and 
\begin{equation}
\label{eq_app_Holder_ineq}
\|uv\|_{C^{0,\alpha}(\overline{U})}\le C\big( \|u\|_{C^{0,\alpha}(\overline{U})}\|v\|_{L^\infty(U)}+ \|u\|_{L^\infty(U)}\|v\|_{C^{0,\alpha}(\overline{U})}\big), \quad u, v\in C^{0,\alpha}(\overline{U}),
\end{equation}
see \cite[Theorem A.7]{Hormander_1976}.

In the proofs of Theorem \ref{thm_main} and Theorem \ref{thm_main_2}, we need a version of the stationary phase lemma with quite explicit control of all the constants involved. Furthermore, note that in our applications,  amplitudes have compact supports in a geodesic ball. 
To state the required result,  let $B_r=\{x\in \R^n: |x|<r\}$ be an open ball centered at $0$ of radius $r>0$, and let $a\in C^{0,\alpha}(\overline{B_r})$ be such that $0\in \mathrm{supp}(a)$ and $a|_{\p B_r}=0$. 
Let $\Psi\in C^\infty(\overline{B_r};\R)$ be such that 
\begin{equation}
\label{eq_app_1}
\Psi(0)=0, \quad  \Psi'(0)=0, \quad \Psi''(0)>0.
\end{equation}
Thus, we have
\begin{equation}
\label{eq_app_1_-2}
\Psi''(0)\ge cI,
\end{equation}
with some $c>0$. Taylor expanding the phase function $\Psi$ and using \eqref{eq_app_1}, we get
\begin{equation}
\label{eq_app_1_-1}
\Psi(x)=\frac{1}{2}\Psi''(0)x\cdot x+\Psi_3(x),
\end{equation}
where 
\begin{equation}
\label{eq_app_1_0}
  |\Psi_3(x)|\le C|x|^3,\quad x\in B_r,
\end{equation} 
with some $C>0$.

\begin{lem} 
\label{lem_stationary_phase}
Let  $r>0$ be such that   
\begin{equation}
\label{eq_app_bound_r}
Cr\le \frac{c}{4}. 
\end{equation} 
Then for $\lambda\ge 1$, we have 
\begin{equation}
\label{eq_app_2}
\lambda^{\frac{n}{2}}\int_{B_r} e^{-\lambda\Psi(x)}a(x)dx=\frac{(2\pi)^{\frac{n}{2}}}{(\emph{\det}  \Psi''(0))^{1/2} }a(0)+\mathcal{O}_{C,c}(1)\lambda^{-\frac{\alpha}{2}}\|a\|_{C^{0,\alpha}(\overline{B_r})}.
\end{equation}
Here the implicit constant depends on $C$ and $c$ only and is independent of $a$.  
\end{lem}

\begin{proof}
First, let us extend $a$ by zero to $\R^n\setminus B_r$ and still denote this extension by $a$.  By Lemma \ref{lem_extension_by_zero}, as  $a|_{\p B_r}=0$, we conclude  that $a\in C^{0,\alpha}(\R^n)$ and 
\begin{equation}
\label{eq_app_2_a_new}
\|a\|_{C^{0,\alpha}(\R^n)}=\|a\|_{C^{0,\alpha}(\overline{B_r})}.
\end{equation}

Using \eqref{eq_app_1_-1}, we write 
\[
\int_{B_r} e^{-\lambda\Psi(x)}a(x)dx=J_1+J_2+J_3,
\]
where 
\begin{align*}
&J_1:=\int_{\R^n} e^{- \frac{\lambda}{2}\Psi''(0)x\cdot x}a(0)dx, \quad J_2:=\int_{\R^n} e^{- \frac{\lambda}{2}\Psi''(0)x\cdot x}(a(x)-a(0))dx,\\
&J_3:=\int_{B_r} e^{- \frac{\lambda}{2}\Psi''(0)x\cdot x}\big(e^{-\lambda \Psi_3(x)}-1\big)a(x)dx. 
\end{align*}
Making the change of variables $x\mapsto\lambda^{\frac{1}{2}}x$ in the integral $J_1$, we obtain that 
\begin{equation}
\label{eq_app_3}
J_1=a(0)\lambda^{-\frac{n}{2}}\int_{\R^n} e^{- \frac{1}{2}\Psi''(0)x\cdot x}dx=\frac{(2\pi)^{\frac{n}{2}}}{(\det\,  \Psi''(0))^{1/2} }a(0)\lambda^{-\frac{n}{2}}.
\end{equation}
To bound $|J_2|$, first using that $a\in C^{0,\alpha}(\R^n)$ and \eqref{eq_app_2_a_new}, we get 
\begin{equation}
\label{eq_app_4}
|a(x)-a(0)|\le \|a\|_{C^{0,\alpha}(\overline{B_r})}|x|^\alpha, \quad x\in \R^n.
\end{equation}
Then using \eqref{eq_app_4},  \eqref{eq_app_1_-2}, and making the change of variables $x\mapsto\lambda^{\frac{1}{2}}x$,  we obtain that 
\begin{equation}
\label{eq_app_5}
\begin{aligned}
|J_2|&\le  \|a\|_{C^{0,\alpha}(\overline{B_r})}\int_{\R^n}e^{-\lambda\frac{c}{2}|x|^2}|x|^\alpha dx=\|a\|_{C^{0,\alpha}(\overline{B_r})}\lambda^{-\frac{n}{2}-\frac{\alpha}{2}}\int_{\R^n}e^{-\frac{c}{2}|x|^2}|x|^\alpha dx\\
&=\mathcal{O}_c(1)\lambda^{-\frac{n}{2}-\frac{\alpha}{2}}\|a\|_{C^{0,\alpha}(\overline{B_r})}.
\end{aligned}
\end{equation}
To bound $|J_3|$, in view of \eqref{eq_app_1_0}, we first observe that 
\begin{equation}
\label{eq_app_6}
|e^{-\lambda\Psi_3(x)}-1|=\bigg|\int_0^1\frac{d}{dt}e^{-\lambda\Psi_3(x)t} dt  \bigg|\le \lambda |\Psi_3(x)| e^{\lambda|\Psi_3(x)|}\le C\lambda|x|^3 e^{C\lambda|x|^3},
\end{equation}
for  $x\in B_r$. 
Then using  \eqref{eq_app_1_-2}, \eqref{eq_app_6}, \eqref{eq_app_bound_r}, and making the change of variables $x\mapsto\lambda^{\frac{1}{2}}x$, we get 
\begin{equation}
\label{eq_app_7}
\begin{aligned}
|J_3|&\le \int_{B_r}e^{-\lambda\frac{c}{2}|x|^2}C\lambda|x|^3 e^{C\lambda|x|^3}|a(x)|dx\le \|a\|_{C^{0,\alpha}(\overline{B_r})}
 \int_{B_r} e^{\lambda (-\frac{c}{2}+Cr)|x|^2}C\lambda|x|^3 dx\\
 &\le \|a\|_{C^{0,\alpha}(\overline{B_r})} \int_{\R^n} e^{-\lambda \frac{c}{4} |x|^2}C\lambda|x|^3 dx= \|a\|_{C^{0,\alpha}(\overline{B_r})} C \lambda^{-\frac{n}{2}-\frac{1}{2}} \int_{\R^n} e^{- \frac{c}{4} |x|^2}|x|^3 dx\\
 &=\mathcal{O}_{C,c}(1)\lambda^{-\frac{n}{2}-\frac{1}{2}}\|a\|_{C^{0,\alpha}(\overline{B_r})}\le 
\mathcal{O}_{C,c}(1) \lambda^{-\frac{n}{2}-\frac{\alpha}{2}}\|a\|_{C^{0,\alpha}(\overline{B_r})},
\end{aligned}
\end{equation}
for $\lambda\ge 1$.
Combining \eqref{eq_app_3}, \eqref{eq_app_5}, and \eqref{eq_app_7}, we obtain \eqref{eq_app_2}. 
\end{proof}

\begin{rem}
\label{rem_stationary_phase}
Lemma \ref{lem_stationary_phase} also holds under the following assumptions: $a \in C^{0,\alpha}(\overline{B_r})$ is such that $\mathrm{supp}(a) \subset B_r$ is compact, $0 \in \mathrm{supp}(a)$, $\Psi \in C^\infty(\mathrm{supp}(a);\R)$ satisfies \eqref{eq_app_1}, and there exists a constant $C > 0$ such that \eqref{eq_app_1_0} holds for all $x \in \mathrm{supp}(a)$.
\end{rem}

The following result is standard and is presented here for the completeness and convenience of the reader only. 
\begin{lem}
\label{lem_extension_by_zero}
Let $a\in C^{0,\alpha}(\overline{B_r})$ be such that $a|_{\partial B_r}=0$. The extension of $a$ to $\mathbb{R}^n\setminus B_r$ by zero, denoted also as $a$, belongs to the class $C^{0,\alpha}(\mathbb{R}^n)$, and we have $\|a\|_{C^{0,\alpha}(\mathbb{R}^n)}=\|a\|_{C^{0,\alpha}(\overline{B_r})}$.
\end{lem}
\begin{proof}
Let $L=\sup_{x, y\in B_r, x\ne y}\frac{|a(x)-a(y)|}{|x-y|^\alpha}$. Since $a\in C(\overline{B_r})$, we have 
\begin{equation}
\label{app-Holder}
|a(x)-a(y)|\le L|x-y|^\alpha,   
\end{equation}
for all $x,y\in \overline{B_r}$. Let $x,y\in \R^n$, and let us show \eqref{app-Holder}. First if both $x,y\in B_r$ or both $x,y\in \mathbb{R}^n\setminus B_r$, \eqref{app-Holder} clearly holds.  The only case we have to consider is when $x\in B_r$ and $y\in \mathbb{R}^n\setminus B_r$. In this case, we consider the straight-line segment connecting $x$ and $y$.  It intersects $\partial B_r$ at some point $z\in \p B_r$ so that $|x-z|\le |x-y|$. Using \eqref{app-Holder}, we get 
\[
|a(x)-a(y)|=|a(x)-a(z)|\le L|x-z|^\alpha\le L|x-y|^\alpha.
\]
The result follows. 
\end{proof}

\section{Boundary determination}

\label{app_boundary_determination}

In the proofs of Theorem \ref{thm_main} and Theorem \ref{thm_main_2} we need the following boundary determination result, which is essentially known, see \cite{Brown_2001}, \cite{Brown_Salo_2006}, \cite[Appendix]{Guillarmou_Tzou_2011},  \cite[Appendix A]{FKOU_2021}, \cite{Krup_Uhlmann_nonlinear_mag} and references therein, and presented here for the completeness and convenience of the reader.

\begin{prop}
\label{prop_app_boundary_deter}
Let $(M,g)$ be a compact smooth Riemannian manifold of dimension $n\ge 2$ with smooth boundary. Let $q\in C(M)$, $\lambda>0$, and assume that $\lambda^2$ is not a Dirichlet eigenvalue of $-\Delta_g$ in $M$. If 
\begin{equation}
\label{eq_app2_1}
 \int_M q u_1u_2u_3u_4dV_g=0,
\end{equation}
for all $u_j\in C^{2,\alpha}(M)$ such that 
\begin{equation}
\label{eq_app2_2}
(-\Delta_g-\lambda^2)u_j=0\quad \text{in} \quad M^{\emph{\text{int}}}, 
\end{equation}
$j=1,\dots, 4$, then $q|_{\p M}=0$. 
\end{prop}

\begin{proof}
We shall follow \cite{Brown_2001}, \cite{Brown_Salo_2006}, constructing an explicit family of functions $v_\mu$, whose boundary values have a highly oscillatory behavior as $\mu\to 0$, while becoming increasingly concentrated near a given point on the boundary of $M$, and we shall follow \cite[Appendix A]{FKOU_2021}, correcting $v_\mu$ into an exact solution to \eqref{eq_app2_2}. 

Let $x_0\in \p M$ and let $(x_1,\dots, x_n)$ be the boundary normal coordinates centered at $x_0$ so that in these coordinates, $x_0 =0$, the boundary $\p M$ is given by $\{x_n=0\}$, and $M^{\text{int}}$ is given by $\{x_n > 0\}$. In the local coordinates, $T_{x_0}\p M=\R^{n-1}$, equipped with the Euclidean metric. 
The unit tangent vector $\tau$ is then given by $\tau=(\tau',0)$ where $\tau'\in \R^{n-1}$, $|\tau'|=1$.  
Let $\eta\in C^\infty_0(\R^n;\R)$ be such that $\mathrm{supp}(\eta)$ is in a small neighborhood of $0$, $\eta(0)\ne 0$,  and 
\[
\int_{\R^{n-1}}\eta(x',0)^2dx'=1.
\]
Let $\frac{1}{3}\le \alpha\le \frac{1}{2}$, and following \cite{Brown_Salo_2006}, we let 
\begin{equation}
\label{eq_app2_3_0}
v(x;\mu)= \mu^{-\frac{\alpha(n-1)}{2}-\frac{1}{2}}\eta\bigg(\frac{x}{\mu^{\alpha}}\bigg)e^{\frac{i}{\mu}(\tau'\cdot x'+ ix_n)}, \quad 0<\mu\ll 1,
\end{equation}
 in the boundary normal coordinates,  so that  $v\in C^\infty(M)$, with $\mathrm{supp}(v)$ in  $\mathcal{O}(\mu^{\alpha})$ neighborhood of $x_0=0$. Here $\tau'$ is viewed as a covector.  A direct computation gives that 
\begin{equation}
\label{eq_app2_3}
\|v(\cdot;\mu)\|_{L^2(M)}=\mathcal{O}(1),  
\end{equation}
as $\mu\to 0$. We look for a solution $u_1$ to \eqref{eq_app2_2} in the form
\begin{equation}
\label{eq_app2_3_1}
u_1(\cdot;\mu)=v(\cdot;\mu)+r(\cdot;\mu),
\end{equation}
 where $r$  is  the solution to the Dirichlet problem, 
\begin{equation}
\label{eq_app2_4}
\begin{cases}
(-\Delta_g-\lambda^2) r= (-\Delta_g-\lambda^2)  v& \text{in}\quad M^{\text{int}},\\
r|_{\p M}=0. 
\end{cases}
\end{equation}
By boundary elliptic regularity, we have $r\in H^1_0(M^{\text{int}})\cap C^\infty(M)$. By \cite[Theorem A.3]{FKOU_2021}, we get 
\begin{equation}
\label{eq_app2_5}
\begin{aligned}
\|r\|_{L^2(M)}\le \|r\|_{H^{1/2+\varepsilon}(M^{\text{int}})}\le C\|(-\Delta_g-\lambda^2)  v\|_{H^{-3/2+\varepsilon}(M^{\text{int}})}\\
\le C\|\Delta_g  v\|_{H^{-3/2+\varepsilon}(M^{\text{int}})}+C\lambda^2\|v\|_{H^{-1}(M^{\text{int}})},
\end{aligned}
\end{equation}
for $0<\varepsilon<1/2$. 
It is shown in \cite{FKOU_2021} that with $\alpha=1/3$ and $\varepsilon=1/12$, 
\begin{equation}
\label{eq_app2_6}
\|\Delta_g v\|_{H^{-3/2+\varepsilon}(M^{\text{int}})}=\mathcal{O}(\mu^{1/12}),
\end{equation}
see formula (A.14) in \cite{FKOU_2021}, 
and that 
\begin{equation}
\label{eq_app2_7}
\|v\|_{H^{-1}(M^{\text{int}})}=\mathcal{O}(\mu^{1-\alpha})=\mathcal{O}(\mu^{2/3}),
\end{equation}
see formula (A.19) in \cite{FKOU_2021} and discussion after it. Combining \eqref{eq_app2_5}, \eqref{eq_app2_6}, and \eqref{eq_app2_7}, we get 
\begin{equation}
\label{eq_app2_8}
\|r\|_{L^2(M)}=\mathcal{O}(\mu^{1/12}). 
\end{equation}
Substituting $u_1$ given by \eqref{eq_app2_3_1} and $u_2=\overline{u_1}$ into the integral identity \eqref{eq_app2_1}, and using \eqref{eq_app2_3},  \eqref{eq_app2_8},  we obtain that 
\begin{equation}
\label{eq_app2_9}
\int_M(qu_3u_4)|v|^2dV_g=-\int_M(qu_3u_4)(v\overline{r}+r\overline{v}+|r|^2)dV_g=o(1),
\end{equation}
as $\mu\to 0$. It is computed in \cite{FKOU_2021}, see formula (A.24), that 
\begin{equation}
\label{eq_app2_10}
\lim_{\mu\to 0}\int_M(qu_3u_4)|v|^2dV_g=\frac{1}{2}q(0)u_3(0)u_4(0). 
\end{equation}
It follows from \eqref{eq_app2_9} and \eqref{eq_app2_10} that 
\begin{equation}
\label{eq_app2_11}
q(0)u_3(0)u_4(0)=0, 
\end{equation}
where $u_3, u_4\in C^{2,\alpha}(M)$ are any solutions of \eqref{eq_app2_2}. Letting $u_3(\cdot; h)=v(\cdot;h)+r(\cdot;h)$, $0<h\ll 1$, where $v(\cdot;h)$ is given by \eqref{eq_app2_3_0} with $\mu$ replaced by $h$, $r(\cdot;h)$ is the solution to \eqref{eq_app2_4} with $v=v(\cdot;h)$, and letting $u_4=\overline{u_3}$, we get from \eqref{eq_app2_11} that 
$q(0)=0$. Here we have used that $\eta(0)\ne 0$. 
\end{proof}

\section{Example of an infinite set of linearly independent admissible perturbations}
\label{sec_example}

Let $M$ be a smooth compact Riemannian manifold of dimension $n\ge 2$ with boundary. We write $C^{0,\alpha}(M)$, $0<\alpha<1$, for the space of H\"older continuous functions on $M$, equipped with the norm,
\[
\|\varphi\|_{C^{0,\alpha}(M)}=\|\varphi\|_{L^\infty(M)}+\sup_{x,y\in M, x\ne y}\frac{|\varphi(x)-\varphi(y)|}{d(x,y)^\alpha},
\]
where $d(x,y)$ is the Riemannian distance between $x$ and $y$, see \cite[Chapter 2, \S 3, \S 9]{Aubin_book_1982}, \cite[Section 13.8]{Taylor_PDE_III}. We define the $C^1$-norm on $M$ by
\[
\|\varphi\|_{C^1(M)}=\|\varphi\|_{L^\infty(M)}+\|\nabla_g \varphi(x)\|_{L^\infty(M)},
\]
where $\nabla_g$ is the gradient with respect to the Riemannian metric $g$. We have the continuous embedding $C^1(M)\subset C^{0,\alpha}(M)$, $0<\alpha<1$, i.e. there exists $B>0$ such that 
\begin{equation}
\label{app_exam_1}
\|\varphi\|_{C^{0,\alpha}(M)}\le B\|\varphi\|_{C^1(M)}, \quad \varphi\in C^1(M).
\end{equation}

In what follows we refer to the definition and properties of simple manifolds to \cite[Section 3.8]{PSU_book}. We have the following result. 
\begin{prop}
Let $(M,g)$ be simple. The set of admissible perturbations $\mathcal{A}(2B)$, with $B$ given in \eqref{app_exam_1}, contains an infinite set of linearly independent admissible perturbations.  
\end{prop}

\begin{proof}
First we have $(M,g)\subset\subset (U,g)$ where $(\overline{U},g)$ is simple. Let $\omega\in U\setminus M$ and let $(r,\theta)$, $r>0$ and $\theta\in \mathbb{S}^{n-1}$, be the polar normal coordinates in $(U,g)$ with center $\omega$. In these coordinates the metric $g$  has the form,
\begin{equation}
\label{app_exam_2}
g(r,\theta)=\begin{pmatrix}
1& 0\\
0& g_1(r,\theta)
\end{pmatrix},
\end{equation}
where $g_1(r,\theta)$ is some $(n-1)\times(n-1)$ positive definite matrix. Consider the polynomial
\[
p(r, \theta)=\sum_{k=0}^N a_k r^k \in C^\infty(M),
\]
where $a_k\ge 0$, $k=0,1,\dots N$,  
\begin{equation}
\label{app_exam_3}
k a_k\le a_{k-1}, \quad k=1,\dots N,
\end{equation}
 and $N\in \N$. It follows from \eqref{app_exam_2} that $\nabla_g p=(\p_r p, 0, \dots, 0)$. Thus, using \eqref{app_exam_2} and \eqref{app_exam_3}, we get
 \[
 |\nabla_g p|_g= |\p_r p|=\sum_{k=1}^N a_k k  r^{k-1}\le \sum_{k=1}^N a_{k-1} r^{k-1}\le p.
 \] 
Hence, $\|p\|_{C^1(M)}\le 2 \|p\|_{L^\infty(M)}$. Combining this with \eqref{app_exam_1}, we get $\|p\|_{C^{0,\alpha}(M)}\le 2 B \|p\|_{L^\infty(M)}$, showing that $p\in \mathcal{A}(2B)$. Thus, in particular, we have $\{\sum_{k=0}^N \frac{r^k}{k!}:N=1,2,\dots\}\subset  \mathcal{A}(2B)$.

\end{proof}

\section{Discussion of the geometric assumption (H1) in Theorem \ref{thm_main_2}}

\label{sec_discussion_examples}

\subsection{Local result} The following result demonstrates that locally, condition (iv) in assumption (H1) can be derived from condition (iii).
\begin{prop}
\label{prop:sec_curvatures}
Let $(N, g)$ be a smooth compact Riemannian manifold without boundary whose sectional curvatures are bounded from above by $\kappa > 0$. Let $\rho < \min(\mathrm{Conv}(N), \frac{\pi}{\sqrt{\kappa}})$, where $\mathrm{Conv}(N) > 0$ is the convexity radius of $N$. There is a constant $C = C(N, g, \rho)$ such that for all $x_0 \in N$ and all pairs of geodesics $\gamma, \eta: (-\rho, \rho) \to N$ with $\gamma(0) = \eta(0) = x_0$,
\[
	d(\gamma(t), \eta(\tau)) \geq C|t|\sin\theta, \quad d(\gamma(t), \eta(\tau)) \geq C|\tau |\sin\theta,  \quad t, \tau \in (-\rho, \rho),
\]
where $\theta$ is the angle of intersection between $\gamma$ and $\eta$.
\end{prop}
\begin{proof}
Let $(x^i)$ denote normal coordinates around $x_0$ with respect to the metric $g$ on $U = B_\rho(x_0)$. In these coordinates, the geodesics $\gamma$ and $\eta$ are simply given by lines through the origin. Let $\alpha:[0, T] \to U$ be the shortest path between $\gamma(t)$ and $\eta(\tau)$, which is unique and entirely contained in $U$ since $\rho < \mathrm{Conv}(N)$. The distance between $\gamma(t)$ and $\eta(\tau)$ is then
\[
	d_g(\gamma(t), \eta(\tau)) = \int_0^T |\dot{\alpha}(t)|_g dt
\]
We parametrise $\alpha$ by arclength with respect to the coordinates $(x^i)$ equipped with the standard flat metric $g_0$ of $\R^n$, so that $|\dot{\alpha}(t)|_{g_0} = 1$. With respect to $g_0$, the shortest path from $\gamma(t)$ to $\eta$ intersects $\eta$ at a right angle. Therefore,
\[
	T = d_{g_0}(\gamma(t), \eta(\tau)) \geq d_{g_0}(\gamma(t), \eta) = |t| \sin \theta
\]
for all $t \in (-\rho, \rho)$. Let $g_\kappa$ be a metric of constant sectional curvature $\kappa$ on $U$ which exists since $\rho < \frac{\pi}{\sqrt{\kappa}}$. The manifold $(U, g_\kappa)$ is isometric to a spherical cap, and hence $g_\kappa$ can be extended to the closure of $U$. By compactness, there is a constant $C = C(\kappa) > 0$ such that $g_\kappa \geq C g_0$ as bilinear forms. By \cite[Theorem 11.10]{Lee_Riem_mfld_2018}, we have $g \geq g_\kappa$ on $U$ since $\rho < \frac{\pi}{\sqrt{\kappa}}$. Therefore,
\[
	d_g(\gamma(t), \eta(\tau)) \geq \int_0^T |\dot{\alpha}(t)|_{g_\kappa} dt \geq C\int_0^T |\dot{\alpha}(t)|_{g_0} dt \geq C |t| \sin\theta.
\]
\end{proof}

\subsection{Discussion related to Example  \ref{example_simple_manifold_H1}}
Let $(M,g)$ be a simple Riemannian manifold. We shall demonstrate that  it satisfies the geometric assumption
(H1). Indeed, consider the bundle
\[
	SM \times_\pi SM = \{((x, v), (y, w)) \in SM \times SM : x = y\}.
\]
As $M$ is a simple manifold, by applying the stopped geodesic flow in backward and forward time, we can see $SM \times_\pi SM$ as a parameter space for pairs of geodesics on $M$ that intersect at a single point. Indeed, to each pair  $((x,v), (x,w)) = (x,v,w)$, we associate the two non-tangential geodesics $\gamma_{x,v}$ and $\gamma_{x, w}$.  

For $0 < \theta_0 < \pi/2$, let $E_{\theta_0} \subset SM \times_\pi SM$ be the set of pairs $(x,v,w)$ such that $|g_x(v, w)| \geq \cos \theta_0$. The set $E_{\theta_0}$ is compact and consists of geodesics that intersect at an angle $\theta \in [\theta_0, \pi/2]$. 

For each $x_0 \in M^{\text{int}}$, choose a pair $(x_0, v, w) \in E_{\theta_0}$ and let $\gamma = \gamma_{x_0, v}$ and $\eta = \gamma_{x_0, w}$. Let us show that such pairs of geodesics satisfy the assumption (H1).

First, $\gamma(0) = \eta(0) = x_0$  and so (i) holds (with shifted times so that $t_0 = \tau_0 = 0$). As $M$ is simple, it is non-trapping and there exists a $T > 0$ such that all geodesics have length at most $T$ so (ii) holds. By our choice of pairs within $E_{\theta_0}$, (iii) holds with $\theta_0$ as above. 

It remains to show (iv). Embed $(M,g)$ in a smooth compact Riemannian manifold without boundary $(N, g)$ whose sectional curvatures are bounded above by $\kappa > 0$. Let $\rho$ be as in Proposition \ref{prop:sec_curvatures} and consider the function $F: E_{\theta_0} \to [0, \infty]$ given by
\[
F(x, v, w) = \dist(\gamma_{x,v}(I_{x,v} \setminus (-\rho, \rho)), \gamma_{x,w}(I_{x,w}))
\]
where $I_{x,v}$ is the closed interval on which $\gamma_{x,v}$ is defined in $M$, $F(x,v,w) = \infty$ if $I_{x,v} \subset (-\rho, \rho)$, and $\mathrm{dist}(A, B) = \inf \{d_g(a, b) : a \in A, b \in B\}$ for any pair of subsets $A$ and $B$ of $M$. By continuous dependence of the geodesics $\gamma_{x,v}$ on $(x,v) \in SM$, the set 
\[
	V_\rho = \{(x,v,w) \in E_{\theta_0} : I_{x,v} \subset (-\rho, \rho)\}
\]
is open and $F$ is continuous on $E_{\theta_0} \setminus V_\rho$. It follows that the sets $F^{-1}((y, \infty])$ are open for all $y \in \R$, that is, $F$ is lower semicontinuous.

The function $(x,v,w) \mapsto \min(F(x,v,w), F(x,w,v))$ is then also lower semicontinuous. Therefore, by compactness of $E_{\theta_0}$ and the fact that geodesics in $M$ intersect at most once, there is $c = c(M, g, \rho) > 0$ such that
\[
	d(\gamma_{x,v}(t), \gamma_{x,w}(\tau)) \geq c
\]
whenever $|t| \geq \rho$ or $|\tau| \geq \rho$ for all $(x,v,w) \in E_{\theta_0}$. This holds because $F(x, v, w) \geq c$ if and only if $d(\gamma_{x,v}(t), \gamma_{x, w}(\tau)) \geq c$ for all $|t| \geq \rho$. Hence, by choosing $r < \min(c, \text{Inj}(M)/2)$, we see that if $d(\gamma(t), \eta(\tau))  \leq r$, then $|t| < \rho$ and $|\tau| < \rho$. Proposition \ref{prop:sec_curvatures} then guarantees there is $C=C(M,g,\rho) > 0$ such that
\[
	|t| \leq \frac{d(\gamma(t), \eta(\tau))}{C \sin \theta_0}, \quad |\tau| \leq \frac{d(\gamma(t), \eta(\tau)) }{C \sin \theta_0}.
\]
Here we have also used that $\theta \in [\theta_0, \pi/2]$. Condition (iv) follows from choosing $r$ as above and $c_0 = (C \sin \theta_0)^{-1}$. Since $\rho$ only depends on $(M, g)$, the constants $r$, $c_0$, and $\theta_0$ can be chosen uniformly.

\subsection{Discussion related to Example \ref{example_cylinder_H1} }
Let $M=\mathbb{S}^1\times[0,a]$, $a>0$, be the cylinder with the usual flat metric. We shall show it satisfies the geometric assumption (H1). Indeed, 
let $x_0 = (e^{it_0}, s_0)\in M^{\text{int}}$ and for $\theta \in (0, \pi/2)$, consider the geodesics
	\begin{align*}
		\gamma_{x_0}(t) &= (e^{it_0}, t), \quad t \in [0, a], \\
		\eta^{\theta}_{x_0}(\tau) &= (e^{i(t_0 - s_0 \tan \theta + \tau \sin \theta)}, \tau \cos\theta), \quad \tau \in \left[0, \frac{a}{\cos \theta}\right].
	\end{align*}
The geodesics $\gamma_{x_0}$ and $\eta_{x_0}^\theta$ intersect at the point $x_0 = \gamma_{x_0}(s_0) = \eta_{x_0}^\theta(s_0/\cos\theta)$ with an intersection angle of $\theta$.  If $|- s_0 \tan \theta + \tau \sin \theta|<2\pi$ for all 
$\tau \in \left[0, \frac{a}{\cos \theta}\right]$, the geodesics $\gamma_{x_0}$ and $\eta_{x_0}^\theta$ do not intersect at any other points. This condition is guaranteed if we require 
\[
|- s_0 \tan \theta + \tau \sin \theta|\le 2 a\tan\theta<2\pi, 
\]
which implies that $0 < \theta < \arctan(\frac{\pi}{a})$. 

Let $\theta_0 = \frac{1}{4} \arctan(\frac{\pi}{a})$. If $\theta \in [\theta_0, 2\theta_0)$, then $\eta^\theta_{x_0}$ rotates at most half a turn around the cylinder. Unwrapping that half of the cylinder into $[0, \pi] \times [0, a]$, we see that the distance between $\gamma_{x_0}(t)$ and $\eta^\theta_{x_0}(\tau)$ is the same as the Euclidean distance. The shortest path from $\gamma_{x_0}(t)$ to the curve $\eta^\theta_{x_0}$ intersects $\eta^\theta_{x_0}$ at a right angle and therefore $d(\gamma_{x_0}(t), \eta^\theta_{x_0}) = |t - s_0| \sin \theta$ since $\gamma_{x_0}(s_0) = x_0$. Similarly, $d(\gamma, \eta^\theta_{x_0}(\tau)) = |\tau - s_0/\cos \theta| \sin \theta$ since $\eta^\theta_{x_0}(s_0/\cos \theta) = x_0$. Therefore,
\[
	d(\gamma_{x_0}(t), \eta_{x_0}^\theta(\tau)) \geq \sin \theta \max(|t - s_0|,|\tau - s_0/\cos \theta|),
\]
for all $t \in [0, a]$ and $ \tau \in \left[0, \frac{a}{\cos \theta}\right]$. 

Therefore, if we choose a pair of geodesics $\gamma = \gamma_{x_0}$ and $\eta = \eta_{x_0}^\theta$ with $\theta \in [\theta_0, 2\theta_0)$ for every point $x_0 \in M^{\text{int}}$, we get a family of pairs of geodesics that satisfies assumption (H1) with $\theta_0$ as above, $T = \frac{a}{\cos(2\theta_0)}$, $0 < r < \pi/2$, and $c_0 = (\sin \theta_0)^{-1}$.

\end{appendix}

\section*{Acknowledgements}

The research of K.K.~is partially supported by the National Science Foundation (DMS 2109199). The research of S.M.~is partially supported by the NSF of China under the grant No.~12301540. S.K.S is partially supported by the Academy of Finland (Centre of Excellence in Inverse Modelling and Imaging, grant 284715) and by the European Research Council under Horizon 2020 (ERC CoG 770924). The research of M.S.~is supported by Research Council Finland (CoE in Inverse Modelling and Imaging and FAME flagship, grants 353091 and 359208) and by the European Research Council under Horizon 2020 (ERC CoG 770924). The research of S.S.-A.~is partially supported by the Natural Sciences and Engineering Research Council of Canada.


\begin{thebibliography}{99}

\bibitem{Angulo-Ardoy_2017}
Angulo-Ardoy, P., \emph{On the set of metrics without local limiting Carleman weights}, 
Inverse Probl. Imaging \textbf{11} (2017), no. 1, 47--64. 

\bibitem{Aubin_book_1982}
Aubin, T., \emph{Nonlinear analysis on manifolds. Monge-Amp\`ere equations}, Springer-Verlag, New York, 1982.



\bibitem{Brown_2001}
Brown,  R., \emph{Recovering the conductivity at boundary from Dirichlet to Neumann map: a pointwise result},  J. of Inverse and ill-posed problems \textbf{9} (2001), 567--574.

\bibitem{Brown_Salo_2006}
Brown, R.,  Salo, M., \emph{Identifiability at the boundary for first-order terms}, 
Appl. Anal. \textbf{85} (2006), no. 6-7, 735--749. 

\bibitem{Calderon_1980}
Calder\'on, A., \emph{On an inverse boundary value problem}, Seminar on Numerical Analysis and its Applications to Continuum Physics (Rio de Janeiro, 1980), pp. 65--73, Soc. Brasil. Mat., Rio de Janeiro, 1980.

\bibitem{CFO_2022}
C\^{a}rstea, C., Feizmohammadi, A., Oksanen, L., \emph{Remarks on the anisotropic Calder\'on problem}, Proc. Amer. Math. Soc. 151 (2023), no. 10, 4461--4473.

\bibitem{do_Carmo_1992}
do Carmo, M., \emph{Riemannian geometry}, Mathematics: Theory \& Applications. Birkh\"auser Boston, Inc., Boston, MA, 1992.

\bibitem{DKSaloU_2009} 
Dos Santos Ferreira, D.,  Kenig, C., Salo, M., Uhlmann, G.,  \emph{Limiting Carleman weights and anisotropic inverse problems},  Invent. Math. \textbf{178} (2009), no. 1, 119--171.




\bibitem{DKuLLS_2020}
Dos Santos Ferreira, D.,  Kurylev, Y.,  Lassas, M., Liimatainen, T.,   Salo, M., \emph{The linearized Calder\'on problem in transversally anisotropic geometries}, Int. Math. Res. Not. \textbf{IMRN} 2020, no. 22, 8729--8765.


\bibitem{DKurylevLS_2016}
Dos Santos Ferreira, D., Kurylev, Y., Lassas, M., Salo, M., \emph{The Calder\'on problem in transversally anisotropic geometries},  J. Eur. Math. Soc. (JEMS) \textbf{18} (2016), no. 11, 2579--2626. 



\bibitem{FKianOksanen_2023}
Feizmohammadi, A., Kian, Y., Oksanen, L., \emph{Rigidity of inverse problems for nonlinear elliptic equations on manifolds}, preprint, \textsf{https://arxiv.org/abs/2306.05839}.


\bibitem{FKOU_2021}
Feizmohammadi, A.,  Krupchyk, K., Oksanen, L., Uhlmann, G.,  \emph{Reconstruction in the Calder\'on problem on conformally transversally anisotropic manifolds}, J. Funct. Anal. \textbf{281} (2021), no. 9, Paper No. 109191, 25 pp.

\bibitem{FLL_2021}
Feizmohammadi, A., Liimatainen, T., Lin, Y.-H., \emph{An inverse problem for a semilinear elliptic equation on conformally transversally anisotropic manifolds}, Ann. PDE \textbf{9} (2023), no. 2, Paper No. 12.

\bibitem{Feizmohammadi_Oksanen_2020}
Feizmohammadi, A.,  Oksanen, L., \emph{An inverse problem for a semi-linear elliptic equation in Riemannian geometries},  J. Differential Equations \textbf{269} (2020), no. 6, 4683--4719.



\bibitem{Gilbarg_Trudinger_book}
Gilbarg, D., Trudinger, N., \emph{Elliptic partial differential equations of second order}, Reprint
of the 1998 edition. Classics in Mathematics. Springer-Verlag, Berlin, 2001.

\bibitem{Guillarmou_Sa_Barreto_2009}
Guillarmou, C.,  S\'a Barreto, A., \emph{Inverse problems for Einstein manifolds},  Inverse Probl. Imaging \textbf{3} (2009), no. 1, 1--15.

\bibitem{Guillarmou_Tzou_2011}
Guillarmou, C.,  Tzou, L.,  \emph{Calder\'on inverse problem with partial data on Riemann surfaces},  Duke Math. J. \textbf{158} (2011), no. 1, 83--120.



\bibitem{Hormander_1976}
H\"ormander, L., \emph{The boundary problems of physical geodesy}, Arch. Rational Mech. Anal.
\textbf{62} (1976), no. 1, 1--52.

\bibitem{KKL01}
Katchalov, A., Kurylev, Y., Lassas, M.,
\emph{Inverse boundary spectral problems}, Chapman \& Hall/CRC Monographs and Surveys in Pure and Applied Mathematics.
CRC Press, Boca Raton, 2001.

\bibitem{KrLiimSalo_2022}
Krupchyk, K.,  Liimatainen, T.,  Salo, M., \emph{Linearized Calder\'on problem and exponentially accurate quasimodes for analytic manifolds},  Adv. Math. \textbf{403} (2022), Paper No. 108362, 43~pp.

\bibitem{Krup_Uhlmann_nonlinear_mag}
Krupchyk, K., Uhlmann, G., \emph{Inverse problems for nonlinear magnetic Schr\"odinger equations on conformally transversally anisotropic manifolds}, Anal. PDE 16 (2023), no. 8, 1825--1868.

\bibitem{Krup_Uhlmann_nonlinear_mag_2018}
Krupchyk, K., Uhlmann, G., \emph{Inverse problems for magnetic Schr\"odinger operators in transversally anisotropic geometries},  Comm. Math. Phys. \textbf{361} (2018), no. 2, 525--582.


\bibitem{Krup_Uhlmann_Yan_nonlinear_mag_2021}
Krupchyk, K., Uhlmann, G., Yan, L.,  \emph{A remark on inverse problems for nonlinear magnetic Schr\"odinger equations on complex manifolds}, Proc. Amer. Math. Soc., to appear.


\bibitem{Lafontaine_Spence_Wunsch_2021}
Lafontaine, D., Spence, E.,  Wunsch, J.,  \emph{For most frequencies, strong trapping has a weak effect in frequency-domain scattering},  Comm. Pure Appl. Math. \textbf{74} (2021), no. 10, 2025--2063.

\bibitem{LLLS_2021}
Lassas, M.,  Liimatainen, T., Lin, Y.-H.,  Salo, M., \emph{Inverse problems for elliptic equations with power type nonlinearities},  J. Math. Pures Appl. (9) \textbf{145} (2021), 44--82.

\bibitem{Lassas_Taylor_Uhlmann_2003}
Lassas, M.,  Taylor, M., Uhlmann, G., \emph{The Dirichlet-to-Neumann map for complete Riemannian manifolds with boundary},  Comm. Anal. Geom. \textbf{11} (2003), no. 2, 207--221. 


\bibitem{Lassas_Uhlmann_2001}
Lassas, M.,  Uhlmann, G., \emph{On determining a Riemannian manifold from the Dirichlet-to-Neumann map}, 
Ann. Sci. \'Ecole Norm. Sup. (4) \textbf{34} (2001), no. 5, 771--787. 

\bibitem{Lee_Riem_mfld_2018}
Lee, J., \emph{Introduction to Riemannian manifolds}, Second edition, Graduate Texts in Mathematics, 176. Springer, Cham, 2018.

\bibitem{Lee_Uhlmann_1989}
Lee, J., Uhlmann, G., \emph{Determining anisotropic real-analytic conductivities by boundary measurements}, Comm. Pure Appl. Math. \textbf{42} (1989), no. 8, 1097--1112.

\bibitem{Liimatainen_Salo_2012}
Liimatainen, T., Salo, M., \emph{Nowhere conformally homogeneous manifolds and limiting Carleman weights}, 
Inverse Probl. Imaging \textbf{6} (2012), no. 3, 523--530. 

\bibitem{Ma_Sahoo_Salo_2022} 
Ma, S.,   Sahoo, S. K., Salo, M., \emph{The anisotropic Calder\'on problem at large fixed frequency on manifolds with invertible ray transform}, preprint 2022, \textsf{https://arxiv.org/abs/2207.02623}.

\bibitem{Ma_Tzou_2021}
Ma, Y., Tzou, L., \emph{Semilinear Calder\'on problem on Stein manifolds with K\"ahler metric}, Bull. Aust. Math. Soc. \textbf{103} (2021), no. 1, 132--144.

\bibitem{Markus_book_1988}
Markus, A., \emph{Introduction to the spectral theory of polynomial operator pencils}, 
Translated from the Russian.  Translations of Mathematical Monographs, 71. American Mathematical Society, Providence, RI, 1988.

\bibitem{Nachman_1996}
Nachman, A., \emph{Global uniqueness for a two-dimensional inverse boundary value problem}, Ann. of Math. (2) \textbf{143} (1996), no. 1, 71--96.

\bibitem{PSU_book}
Paternain, G.,  Salo, M.,  Uhlmann, G., \emph{Geometric inverse problems--with emphasis on two dimensions},  Cambridge Studies in Advanced Mathematics, 204. Cambridge University Press, Cambridge, 2023

\bibitem{Rakesh_Salo_2020}
 Rakesh; Salo,  M., \emph{Fixed angle inverse scattering for almost symmetric or controlled perturbations},  SIAM J. Math. Anal. \textbf{52} (2020), no. 6, 5467--5499. 


\bibitem{Rakesh_Uhlmann_2014}
Rakesh; Uhlmann, G.,  \emph{Uniqueness for the inverse backscattering problem for angularly controlled potentials}, Inverse Problems \textbf{30} (2014), no. 6, 065005, 24. 
 

\bibitem{salo2017calderon}
Salo, M., \emph{The Calder\'on problem and normal forms}, preprint, \textsf{https://arxiv.org/abs/1702.02136}.

\bibitem{Sjostrand_2009}
Sj\"ostrand, J., \emph{Eigenvalue distribution for non-self-adjoint operators with small multiplicative random perturbations},
Ann. Fac. Sci. Toulouse Math. (6) \textbf{18} (2009), no. 4, 739--795.


\bibitem{Sjostrand_book_2019}
Sj\"ostrand, J., \emph{Non-self-adjoint differential operators, spectral asymptotics and random perturbations},  Pseudo--Differential Operators. Theory and Applications, 14. Birkh\"auser/Springer, Cham, 2019. 


\bibitem{Sjostrand_Zworski_1991}
Sj\"ostrand, J.,  Zworski, M., \emph{Complex scaling and the distribution of scattering poles}, 
J. Amer. Math. Soc. \textbf{4} (1991), no. 4, 729--769. 


\bibitem{St-Amant_2024}
St-Amant, S., \emph{Gaussian beams and Calder\'on problem for connections at large frequency}, preprint 2024. 

\bibitem{Stefanov_2001}
Stefanov, P., \emph{Resonance expansions and Rayleigh waves},  Math. Res. Lett. \textbf{8} (2001), no. 1--2, 107--124.

\bibitem{Tang_Zworski_1998}
Tang, S.-H.,  Zworski, M., \emph{From quasimodes to resonances}, 
Math. Res. Lett. \textbf{5} (1998), no. 3, 261--272.

\bibitem{Taylor_PDE_III}
Taylor, M.,  \emph{Partial differential equations III. Nonlinear equations}, Applied Mathematical Sciences, 117. Springer, New York, 2011.

\bibitem{Uhlmann_2014}
Uhlmann, G., \emph{Inverse problems: seeing the unseen}, Bull. Math. Sci. \textbf{4} (2014), no. 2, 209--279.

\bibitem{Uhlmann_Wang_2021}
Uhlmann, G., Wang, Y., \emph{The anisotropic Calder\'on problem for high fixed frequency},  SIAM J. Math. Anal., to appear. 




\bibitem{Zworski_book}
Zworski, M., \emph{Semiclassical analysis},  Graduate Studies in Mathematics, 138. American Mathematical Society, Providence, RI, 2012.



\end{thebibliography}
\end{document}